\pdfoutput=1

\documentclass[final]{amsart}
\calclayout

\usepackage[utf8]{inputenc}

\usepackage{hyperref}
\hypersetup{
	colorlinks = true,
	linkcolor = {blue},
	urlcolor = {black},
	citecolor = {black}
}

\usepackage{amssymb}
\usepackage{enumitem}
\usepackage[capitalize]{cleveref}
\usepackage{color}

\usepackage[backend=biber,
            style=alphabetic,
            bibencoding=utf8,
            language=auto,
            autolang=other,
            giveninits=true,
            doi=false,
            isbn=false,
            url=false,
            maxnames=10,
            sorting=nty]{biblatex}
\usepackage{asymptote}

\usepackage{mathtools}
\usepackage{tikz}
\usetikzlibrary{babel} 
\usetikzlibrary{cd} 

\crefname{enumi}{}{} 

\newtheorem{theorem}{Theorem}[section]
\newtheorem{lemma}[theorem]{Lemma}
\newtheorem{proposition}[theorem]{Proposition}
\newtheorem{corollary}[theorem]{Corollary}

\theoremstyle{remark}
\newtheorem{remark}{Remark}[section]
\newtheorem*{remark*}{Remark}

\theoremstyle{definition}
\newtheorem{definition}{Definition}[section]

\usepackage[backend=biber,style=alphabetic,bibencoding=utf8,language=auto,autolang=other,giveninits=true,doi=false,isbn=false,url=false,maxnames=10,sorting=nty]{biblatex}

\newcommand{\N}{\ensuremath{\mathbb N}} 
\newcommand{\Z}{\ensuremath{\mathbb Z}} 
\newcommand{\Q}{\ensuremath{\mathbb Q}} 
\newcommand{\R}{\ensuremath{\mathbb R}} 
\providecommand{\C}{}
\renewcommand{\C}{\ensuremath{\mathbb C}} 
\newcommand{\F}{\ensuremath{\mathbb F}} 
\newcommand{\PS}{\ensuremath{\mathcal P}} 

\newcommand{\defeq}{\ensuremath{\coloneqq}}

\newcommand\M{\ensuremath{\mathcal{M}}} 
\newcommand\rk{\ensuremath{\operatorname{rk}}} 
\newcommand{\FS}{\ensuremath{\operatorname{FS}}} 
\newcommand{\im}{\ensuremath{\operatorname{Im}\,}} 
\newcommand{\Hom}{\ensuremath{\operatorname{Hom}}} 
\newcommand{\cyclic}[1]{\ensuremath{\Z / #1\Z}}		
\newcommand{\ord}{\ensuremath{\operatorname{ord}}}		
\newcommand{\emptyparam}{\ensuremath{\,\cdot\,}}

\begin{document}

\title{On The Determination of Sets By Their Subset Sums}

\author{Andrea Ciprietti}
\address{Andrea Ciprietti
\hfill\break Department of Mathematics, University of Pisa, Pisa, Italy}
\email{andreaciprietti99@gmail.com}

\author{Federico Glaudo}
\address{Federico Glaudo
\hfill\break School of Mathematics, Institute for Advanced Study, 1 Einstein Dr., Princeton NJ 05840, U.S.A.}
\email{fglaudo@ias.edu}

\begin{abstract}
    Let $A$ be a multiset with elements in an abelian group. Let $\FS(A)$ be the multiset containing the $2^{|A|}$ sums of all subsets of $A$.
    
    We study the reconstruction problem ``Given $\FS(A)$, is it possible to identify $A$?'', and we give a satisfactory answer for all abelian groups.
    We prove that, up to identifying multisets through a natural equivalence relation, the function $A \mapsto \FS(A)$ is injective (and thus the reconstruction problem is solvable) if and only if every order $n$ of a torsion element of the abelian group satisfies a certain number-theoretical property linked to the multiplicative group $(\cyclic{n})^*$.
        
    The core of the proof relies on a delicate study of the structure of cyclotomic units. Moreover, as a tool, we develop an inversion formula for a novel discrete Radon transform on finite abelian groups that might be of independent interest.
\end{abstract}

\maketitle

\section{Introduction}

Let $G$ be an abelian group and let $A=\{a_1, a_2, \dots, a_{|A|}\}$ be a finite multiset (i.e., a set with repeated elements) with elements in $G$ (see \cref{subsec:multisets} for a formal definition of multiset). Its \emph{subset sums multiset} $\FS(A)$, that is, the multiset containing the $2^{|A|}$ sums over all subsets of $A$ (taking into account multiplicities), is defined as
\begin{equation*}
    \FS(A) \defeq \Big\{\sum_{i\in I} a_i: I\subseteq \{1, 2, \dots, |A|\}\Big\}.
\end{equation*}
We study the following reconstruction question: 
\begin{center}
\textit{If one is given $\FS(A)$, is it possible to identify $A$?} 
\end{center}
As we will see, this strikingly simple question features a rich structure and its solution spans a wide range of mathematics: from the theory of cyclotomic units, to an inversion formula for a novel discrete Radon transform. 
Before going deeper into the problem, let us give some background on related results in the literature.

If, instead of $\FS(A)$, one is given the sums over all the $\binom{|A|}{s}$ subsets with fixed size equal to $s$ (e.g., if $s=2$, the sums over all pairs), the reconstruction problem has been studied in the case of a free abelian group $G=\Z^d$ \cite{SelfridgeStraus1958,GordonFraenkelStraus1962}. For pairs (i.e.\ $s=2$), the reconstruction is possible when the size of $A$ is not a power of $2$ \cite[Theorem 1 and Theorem 2]{SelfridgeStraus1958}. For $s$-subsets with $s>2$, the reconstruction is possible if the size of $A$ does not belong to a finite subset of \emph{bad} sizes \cite[Section 4]{GordonFraenkelStraus1962}. See the recent survey \cite{fomin2019} for a detailed presentation of the history of this problem.

If, instead of $\FS(A)$, one is given $A+A$ (i.e., the sum of any two elements of $A$, not necessarily distinct), the problem has been studied extensively for infinite sets of nonnegative integers (see, for example, \cite{Lev2004,ChenLev2016,Helou2017,KissSandor2019} and the survey \cite{Nathanson2008}).

It might seem that if one is only provided with the sums of $s$-subsets (i.e., subsets with size $s$) then the reconstruction is strictly harder than if one is provided the sums of all subsets. This is not true because the information is not ordered and thus, even if we have more information, it is also harder to determine which value corresponds to which subset.

Let us now go back to the reconstruction problem for $\FS$. The first important observation is the following one. Given a multiset $A$ and a subset $B\subseteq A$ whose sum equals $0$ (i.e.\ $\sum_{b\in B} b = 0$), if we flip the signs of elements of $B$ then $\FS$ does not change. So, if $A'\defeq (A\setminus B)\cup (-B)$, then $\FS(A)=\FS(A')$ (see \cref{fig:sim0-same-fs} for an explanation).

\begin{figure}[h]
    \includegraphics{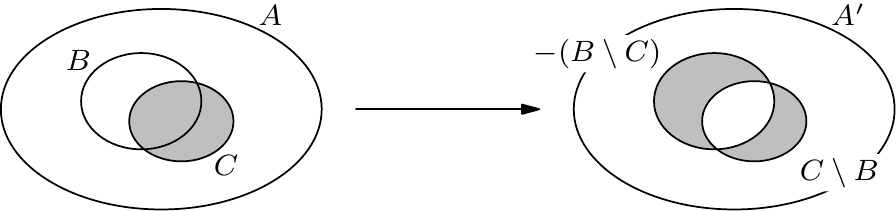}
    \caption{
        Proof by picture of $A\sim_0 A' \implies \FS(A)=\FS(A')$.
        The set $C$ in $A$ (highlighted in gray) and the set $(C\setminus B)\cup(-(B\setminus C))$ in $A'$ (highlighted in gray) have the same sum because the sum of the elements in $B$ is assumed to be $0$. Thus, we have a bijection between the subsets of $A$ and $A'$ which keeps the sum unchanged, hence $\FS(A)=\FS(A')$.
    }
    \label{fig:sim0-same-fs}
\end{figure}

Hence, if we only know $\FS(A)$, the best we can hope for is to identify the equivalence class of $A$ with respect to the following equivalence relation.
\begin{definition}
    Given two multisets $A, A'$ with elements in $G$, we say that $A\sim_0 A'$ if and only if $A'$ can be obtained from $A$ by flipping the signs of the elements of a subset of $A$ with null sum, i.e., if there exists $B\subseteq A$, with $\sum_{b\in B}b=0$, such that $A'=(A\setminus B)\cup(-B)$.
\end{definition}
We have already observed that if $A\sim_0 A'$ then $\FS(A)=\FS(A')$.
If the group is $G=\Z$, this turns out to be an ``if and only if'' (see \cref{prop:multiply-by-Z}), while if $G=\cyclic{2}$ it is not (indeed, in $\cyclic{2}$ one has $\FS(\{0, 1\})=\{0,0,1,1\}=\FS(\{1,1\})$). It is natural to consider the class of abelian groups such that the double implication holds, i.e.\ the fibers of $\FS$ coincide with the equivalence classes of $\sim_0$.
\begin{definition}
    An abelian group $G$ is $\FS$-regular if, for any two multisets $A,A'$ with elements in $G$, it holds $\FS(A)=\FS(A')$ if and only if $A\sim_0 A'$.
\end{definition}
We have already observed that $\cyclic{2}$ is not $\FS$-regular; moreover, any group containing a subgroup that is not $\FS$-regular cannot be $\FS$-regular. The next smallest non-$\FS$-regular group is elusive; in fact, it turns out that $\cyclic{n}$ is $\FS$-regular for $n=3,5,7,9,11,13,15$. But $\cyclic{17}$ is not $\FS$-regular, and then $\cyclic{n}$ is $\FS$-regular for $n=19,21,23,25,27,29$ and not $\FS$-regular for $m=31,33$. These small examples suggest that the $\FS$-regularity of $G$ may be related to the behavior of powers of two in $G$ (notice that $17, 31, 33$ are adjacent to a power of two).

Our main result is the characterization of $\FS$-regular groups. In order to state our result, we need to introduce a subset of the natural numbers.

\begin{definition}
    Let $O_{\FS}$ be the set of odd natural numbers $n\ge 1$ such that $(\cyclic{n})^*$ is covered by $\{\pm 2^j: j\ge 0\}$; more precisely, for each $x\in\Z$ relatively prime with $n$ there exists $j\ge 0$ such that either $x-2^j$ or $x+2^j$ is divisible by $n$.
\end{definition}

\begin{remark*}
    The first few elements of $O_{\FS}$ are
    \begin{equation*}
        O_{\FS} = \{1, 3, 5, 7, 9, 11, 13, 15, 19, 21, 23, 25, 27, 29, 35, 37, 39, 45, 47, 49,53,55,\dots\},
    \end{equation*}
    and the first few \emph{missing} odd numbers are
    \begin{equation*}
        (2\N+1)\setminus O_{\FS} = \{17, 31, 33, 41, 43, 51, 57, 63, 65, 73, 85, 89, 91, 93, 97, 99, 105, \dots\}. 
    \end{equation*} 
    Let us remark that if $n\in O_{\FS}$ then also all divisors of $n$ belong to $O_{\FS}$. Moreover, if $n\in O_{\FS}$ then $n$ has at most two distinct prime factors. We prove these and some other basic properties of the set $O_{\FS}$ at the end of \cref{sec:basic-facts}.
\end{remark*}

We can now state our main theorem.

\begin{theorem}[Characterization of $\FS$-regular groups]\label{thm:main}
    An abelian group $G$ is $\FS$-regular if and only if $\ord(g)\in O_{\FS}$ for all $g\in G$ with finite order.
\end{theorem}

As a tool in the proof of \cref{thm:main} (see \cref{subsec:sketch}) we define a novel discrete Radon transform for abelian groups and we prove an inversion formula for it. We refer to \cref{sec:radon} for some motivation on the definition and for an in-depth discussion of the existing related literature. Since the inversion formula for the Radon transform may have other applications beyond the scope of this paper, we state it here for the interested readers.

\begin{theorem}[Inversion formula for the discrete Radon transform]\label{thm:radon-inversion}
    Let $n, d\ge 1$ be positive integers.
    Given a function $f:(\cyclic{n})^d\to\C$, its discrete Radon transform $Rf=R_{n, d}f:\Hom((\cyclic{n})^d,\, \cyclic{n})\times \cyclic{n}\to\C$ is defined as
    \begin{equation*}
        Rf(\psi, c) = \sum_{x:\,\psi(x)=c} f(x).
    \end{equation*}
    One can reconstruct the values of $f$ from $Rf$ through the formula, valid for any $x\in(\cyclic{n})^d$,
    \begin{equation*}
        f(x)
        =
        \frac{1}{n^{d-1}\varphi(n)}
        \sum_{\psi\in \Hom((\cyclic{n})^d, \cyclic{n})}
        Rf(\psi, \psi(x))\prod_{p\mid \psi}
        (1-p^{d-1}),
    \end{equation*}
    where the notation $p\mid \psi$ shall be understood as the fact that the \emph{prime} $p$, divisor of $n$, divides all the elements in the image of $\psi$, or equivalently that $\psi$ takes values into $p\Z/n\Z$.
\end{theorem}

\subsection{Sketch of the proof and structure of the paper}\label{subsec:sketch}
Let us briefly describe the strategy that the proof follows, postponing a more detailed presentation to the dedicated sections.

For the negative part of the statement, it is sufficient to show that $\cyclic{n}$ is not $\FS$-regular if $n\not\in O_{\FS}$. For this, we construct an explicit counterexample in \cref{prop:fs-regular-implies-ofs}.

Proving that if the orders belong to $O_{\FS}$ then the group is $\FS$-regular is more complicated and relies on some nontrivial properties of the units of cyclotomic fields and on the inversion formula for a novel discrete Radon transform on finite abelian groups.
The proof is divided into three steps.

\vspace{0.5em}
    \textit{Step 1: Proof for $G=\cyclic{n}$.} Through the polynomial identity
    \begin{equation*}
        \sum_{s\in\FS(A)} t^s \equiv \prod_{a\in A}(1+t^a) \pmod{t^n-1},
    \end{equation*}
    we reduce the $\FS$-regularity of $\cyclic{n}$ to the study of the kernel of the map
    \begin{equation*}
        \Z^n \ni x = (x_0, x_1, \dots, x_{n-1}) \mapsto \Big(\prod_{j=0}^{n-1}(1+\omega_d^j)^{x_j}\Big)_{d\mid n},
    \end{equation*}
    where $\omega_d\in \C$ is a $d$-th primitive root of unity and the codomain of the map consists of tuples indexed by the divisors of $n$.
    Thanks to a dimensional argument, identifying the kernel of such map is equivalent to identifying its image, which is exactly what we do in \cref{lem:phi-has-full-rank}. This is the hardest and most technical proof of the whole paper. Up to this point, we have used only that $n$ is odd. The fact that $n\in O_{\FS}$ is needed in the computation of the rank of the image, which relies heavily on the theory of cyclotomic units (see \cref{lem:Kn-rank}).

    This step is carried out in \cref{sec:finite-cyclic}.
    
\vspace{0.5em}
    \textit{Step 2: $\cyclic{n}$ is $\FS$-regular $\implies$ $(\cyclic{n})^d$ is $\FS$-regular.} Take $A,A'$ multisets with elements in $(\cyclic{n})^d$ such that $\FS(A)=\FS(A')$. Given a homomorphism $\psi:(\cyclic{n})^d\to\cyclic{n}$, by linearity, it holds $\FS(\psi(A)) = \FS(\psi(A'))$, and since $\cyclic{n}$ is $\FS$-regular this implies that $\psi(A)\sim_0\psi(A')$. So, we know that $\psi(A)\sim_0\psi(A')$ for all homomorphisms $\psi:(\cyclic{n})^d\to\cyclic{n}$. In order to deduce that $A\sim_0 A'$, we introduce a discrete Radon transform for finite abelian groups (see \cref{def:radon}) and we use its invertibility  to reconstruct a multiset $B\in\M((\cyclic{n})^d)$ from its \emph{projections} $\{\psi(B):\, \psi\in\Hom((\cyclic{n})^d, \, \cyclic{n})\}$.
    
    This step is performed in \cref{sec:radon}.
    
\vspace{0.5em}
    \textit{Step 3: $G$ is $\FS$-regular $\implies$ $G\oplus \Z$ is $\FS$-regular.} In this step, we exploit crucially that $\Z$ is totally ordered. The argument is short and purely combinatorial. This is done in \cref{sec:multiply-by-Z}.
\vspace{0.5em}

Once these three steps are established, \cref{thm:main} follows naturally, as shown in \cref{sec:proof-main-thm}.
Let us remark here that our proof is not constructive, hence it does not provide an efficient algorithm to find the $\sim_0$-equivalence class of $A$ if $\FS(A)$ is known\footnote{The nonconstructive part of the proof is contained \cref{sec:finite-cyclic}. In fact, we show that a certain map is injective by proving its surjectivity and then applying a standard dimension argument. This kind of reasoning does not produce an efficient way to invert the map we have proven to be injective.}.

To make the paper accessible to a broad audience, in \cref{sec:preliminaries} we recall basic facts about multisets, abelian groups, and cyclotomic units.

\subsection*{Acknowledgements}
The authors are thankful to Fabio Ferri for providing valuable suggestions and references about the theory of cyclotomic units, and also to Michele D'Adderio, Elia Bruè. We want to thank Noah Kravitz because his comments led to a considerable improvement of the inversion formula for the discrete Radon transform.
The second author is supported by the National Science Foundation under Grant No. DMS-1926686.

\section{Notation and Preliminaries}\label{sec:preliminaries}

\subsection{Multisets}\label{subsec:multisets}
A \emph{multiset} with elements in a set $X$ is an unordered collection of elements of $X$ which may contain a certain element more than once \cite{blizard1989}. For example, $\{1, 1, 2, 2, 3\}$ is a multiset.
Rigorously, a multiset $A$ is encoded by a function $\mu_A:X\to\Z_{\ge 0}$ ($\Z_{\ge 0}$ denotes the set of nonnegative integers) such that $\mu_A(x)$ represents the multiplicity of the element $x$ in $A$. For example, if $A=\{1, 1, 2, 2, 3\}$ then $\mu_A(1)=2$, $\mu_A(2)=2$, $\mu_A(3) = 1$.

A multiset $A$ is \emph{finite} if $\sum_{x\in X}\mu_A(x) < \infty$. The cardinality of a finite multiset $A\in\M(X)$ is given by $|A|\defeq \sum_{x\in X}\mu_A(x)$.

Given a set $X$, let us denote with $\M(X)$ the family of finite multisets with elements in $X$.

Let us define the usual set operations on multisets. Notice that all of them are the natural generalization of the standard version when one takes into account the multiplicity of elements.
Fix two multisets $A, B\in \M(X)$.
\begin{description}
    \item[Membership] We say that $x\in X$ is an element of $A$, denoted by $x\in A$, if $\mu_A(x) \ge 1$.
    \item[Inclusion] We say that $A$ is a subset of $B$, denoted by $A\subseteq B$, if $\mu_A(x)\le \mu_B(x)$ for all $x\in X$.
    \item[Union] The union $A\cup B\in\M(X)$ is defined as $\mu_{A\cup B}(x) \defeq \mu_A(x) + \mu_B(x)$. Hence, $\{1\}\cup\{1, 2\}=\{1, 1, 2\}$.
    \item[Cartesian product] The Cartesian product $A\times B\in \M(X\times X)$ is defined as $\mu_{A\times B}((x_1, x_2)) = \mu_A(x_1)\mu_B(x_2)$.
    \item[Difference] If $A\subseteq B$, the difference $B\setminus A$ is defined as $\mu_{B\setminus A}(x) \defeq \mu_B(x) - \mu_A(x)$.
    \item[Pushforward] Given a function $f:X\to Y$, the pushforward $f(A)\in\M(Y)$ of the multiset $A$ (denoted also by $\{f(a): a\in A\}$) is defined as
    \begin{equation*}
        \mu_{f(A)}(y) = \sum_{x\in f^{-1}(y)} \mu_A(x).
    \end{equation*}
    \item[Power set] The power set of $A$ (the family of subsets of $A$), denoted by $\PS(A)\in\M(\M(X))$, is a multiset defined recursively as follows. For the empty multiset, we have $\PS(\varnothing) \defeq \{\varnothing\}$; otherwise let $a\in A$ be an element of $A$ and define
    \begin{equation*}
        \PS(A) \defeq \PS(A\setminus\{a\}) \cup \Big\{A'\cup\{a\}:\, A'\in \PS(A\setminus\{a\})\Big\}.
    \end{equation*}
    Notice that $|\PS(A)| = 2^{|A|}$.
    Whenever we iterate over the subsets of $A$ (e.g., $\{f(A'):\, A'\subseteq A\}$ or $\sum_{A'\subseteq A}f(A')$), the iteration has to be understood over $\PS(A)$ (hence the subsets are counted with multiplicity).
    
    \emph{Taking the complement} is an involution of the power set, i.e., $\PS(A)=\{A\setminus A':\, A'\in\PS(A)\}$, and we have the following identity for the power set of a union
    \begin{equation*}
        \PS(A\cup B) = \{A'\cup B':\, (A', B')\in \PS(A)\times \PS(B)\}.
    \end{equation*}
    \item[Sum (and product)] If the set $X$ is an additive abelian group, we can define the sum $\sum A\in X$ of the elements of $A$ as
    \begin{equation*}
        \sum A\defeq \sum_{x\in X}\mu_A(x)x.
    \end{equation*}
    Analogously, if $X$ is a multiplicative abelian group, one can define the product $\prod A$ of the elements of $A$. 
\end{description}

\subsection{Abelian Groups}\label{subsec:abelian-groups}
Let us recall some basic facts about abelian groups that we will use extensively later on.

Any finitely generated abelian group is isomorphic to a finite product of cyclic groups \cite[Chapter I, Section 8]{lang2002}. We denote with $\Z/n\Z$ the cyclic group with $n$ elements.

Given some elements $g_1, g_2, \dots, g_k\in G$ of an abelian group, we denote with $\langle g_1, g_2,\dots, g_k\rangle$ the subgroup generated by such elements.
Given an element $g\in G$, its order (which may be equal to $\infty$) is denoted by $\ord(g)$.

For an abelian group $G$, its rank $\rk(G)$ is the cardinality of a maximal set of $\Z$-independent\footnote{Some elements $g_1, g_2, \dots, g_k\in G$ are $\Z$-independent if, whenever $\sum_i a_ig_i = 0$ for some $a_1, a_2,\dots, a_k\in \Z$, it holds $a_1=a_2=\cdots=a_k=0$.} elements of $G$.
Let us list some useful properties of the rank (see \cite[Chapter I and XVI]{lang2002}).
\begin{itemize}
    \item Any finitely generated abelian group $G$ is isomorphic to $\Z^{\rk(G)}\oplus G'$ where $G'$ is a finite abelian group.
    \item Given two abelian groups $G, H$, it holds $\rk(G\oplus H)=\rk(G)+\rk(H)$.
    \item For a homomorphism $\phi: G \to H$ of abelian groups, it holds $\rk(G) = \rk(\ker \phi) + \rk(\im \phi)$.
    \item An abelian group has null rank if and only if all elements have finite order.
    \item Let $G_1, G_2, G_3$ be three abelian groups and $\phi_1:G_1\to G_2$, $\phi_2:G_2\to G_3$ be two homomorphisms with full rank, i.e.\ $\rk(\im\phi_1)=\rk(G_2)$ and $\rk(\im\phi_2)=\rk(G_3)$. Then $\phi_2\circ\phi_1:G_1\to G_3$ has full rank as well, i.e.\ $\rk(\im\phi_2\circ \phi_1) = \rk(G_3)$
    \item Given an abelian group $G$, let us denote with $G\otimes\Q$ its tensor product (as a $\Z$-module) with $\Q$ (see \cite[Chapter XVI]{lang2002}). The dimension of $G\otimes\Q$ as vector space over $\Q$ coincides with $\rk(G)$.
    \item For a homomorphism $\phi:G\to H$ of abelian groups, let $\phi\otimes\Q:G\otimes\Q\to H\otimes\Q$ be its tensorization with $\Q$. It holds $\rk(\im\phi)=\dim_\Q(\im(\phi\otimes\Q))$.
\end{itemize}

\subsection{Units of cyclotomic fields}\label{subsec:intro_cyclo}

Given $n\ge 1$, let $\omega_n\defeq\exp(2\pi i/n)$ be the primitive $n$-th root of unity with minimum positive argument.

The algebraic number field $\Q(\omega_n)$ is called \emph{cyclotomic field}. It is well-known that the ring of integers of $\Q(\omega_n)$ coincides with $\Z[\omega_n]$.
Our main focus is the group of units of $\Q(\omega_n)$, that consists of the invertible elements of its ring of integers.

For $0<r<n$ and $s\ge 1$ coprime with $n$, the element $\xi\defeq \frac{1-\omega_n^{rs}}{1-\omega_n^r}$ is a unit of $\Q(\omega_n)$.
Indeed $\xi = 1 + \omega_n^r + \cdots + \omega_n^{(s - 1)r}\in\Z[\omega_n]$ and, if $u\in\N$ is such that $n$ divides $us-1$, then
\begin{equation*}
    \xi^{-1} = \frac{1 - \omega_n^{rus}}{1 - \omega_n^{rs}} = 1 + \omega_n^{rs} + \cdots + \omega_n^{(u - 1)rs} \in \Z[\omega_n].
\end{equation*}

It turns out that these units are sufficient to generate a subgroup of finite index of the units of $\Q(\omega_n)$. The following statement follows from \cite[Theorem 8.3 and Theorem 4.12]{washington97}.
\begin{theorem}\label{thm:Cn-def}
    For any odd $n\ge 3$, the multiplicative group $C_n \subseteq\C$ generated by 
    \begin{equation*}
        \Big\{\frac{1-\omega_n^{rs}}{1-\omega_n^r}:\, \text{$0<r<n$, $s\ge 1$ coprime with $n$}\Big\}
    \end{equation*}
    is a subgroup of finite index of the units of $\Q(\omega_n)$.
\end{theorem}

Thus, applying Dirichlet's unit Theorem (see \cite[Theorem 38]{marcus77}), we are able to compute the rank of $C_n$ (since it coincides with the rank of the group of units of $\Q(\omega_n)$).
\begin{corollary}\label{cor:rk-Cn}
    For any odd $n\ge 3$, we have $\rk(C_n)=\frac{\varphi(n)}2-1$, where $\varphi$ is Euler's totient function (and $C_n$ is defined in \cref{thm:Cn-def}).
\end{corollary}

The units of $\Q(\omega_n)$ satisfy a family of nontrivial relations known as distribution relations (see \cite[151]{washington97}). We recall here the relations in the form we will need. Notice that $1+\omega_n^j$ is a unit for $1\le j<n$ because of the identity $1+\omega_n^j=\frac{1-\omega_n^{2j}}{1-\omega_n^j}\in C_n$.

\begin{proposition}[Distribution relations]\label{prop:distribution-relations}
    Let $n \ge 1$ be an odd integer and let $p$ be one of its prime divisors\footnote{The identity holds, with the same proof, also without the assumption that $p$ is prime.}. 
    For any $0 \le j < \frac{n}{p}$, the identity
    \begin{equation*}
        \prod_{k = 0}^{p - 1} (1 + \omega_n^{j + kn/p}) = 1 + \omega_n^{jp}
    \end{equation*}
    holds.
\end{proposition}

\begin{proof}
    The numbers $\{1 + \omega_n^{j + kn/p}\}_{0 \le k < p}$ are the roots of the monic polynomial
    $(t - 1)^p - \omega_n^{jp} \in \C[t]$.
    Therefore, their product equals the constant term of the polynomial multiplied by $(-1)^p$,
    which is $((-1)^p-\omega_n^{jp})(-1)^p = 1 + \omega_n^{jp}$.
\end{proof}

\section{Definitions and basic facts}\label{sec:basic-facts}
In this section we give some fundamental definitions (some of them are already present in the introduction, we repeat them here for the ease of the reader) and we prove one basic result which will be useful multiple times in the paper.
\begin{definition}
    Let $G$ be an additive abelian group and take $A\in\M(G)$.
    The \emph{subset sums multiset} of $A$ is (we adopt the notation of \cite{TaoVu2006}) 
    \begin{equation*}
        \FS(A) \defeq \Big\{\sum B : B \in \PS(A)\Big\},
    \end{equation*}
    that is, the multiset whose elements are the sums of the subsets of $A$.
\end{definition}
When studying the injectivity of $\FS$, one soon notices that if we take a multiset $A\in\M(G)$ and we flip the sign of a subset of its elements with zero sum, obtaining another multiset $A'\in\M(G)$, then the subset sums do not change, i.e.\ $\FS(A)=\FS(A')$. Therefore, the following definition and the results of \cref{lem:basic-properties} should appear natural.

\begin{definition}\label{def:sim}
    Given an additive abelian group $G$, we define the equivalence relations $\sim$ and $\sim_0$ over $\M(G)$ as follows:
    \begin{itemize}
        \item Given $A, A'\in\M(G)$, $A \sim A'$ if $A'$ is obtained from $A$ by changing the sign of the elements of a subset of $A$. More formally, $A\sim A'$ if and only if there exists $B\subseteq A$ such that $A' = (A\setminus B)\cup (-B)$.
        \item Given $A, A'\in\M(G)$, $A \sim_0 A'$ if $A'$ is obtained from $A$ by changing the sign of the elements of a zero-sum subset of $A$. More formally, $A\sim_0 A'$ if and only if there exists $B\subseteq A$ with null sum $\sum B=0_G$ such that $A' = (A\setminus B)\cup (-B)$.
    \end{itemize}
\end{definition}
Notice that the relations $\sim$ and $\sim_0$ are reflective and transitive.

\begin{lemma}\label{lem:basic-properties}
    Given two multisets $A, A'\in\M(G)$ with elements in an abelian group $G$, we have the following statements concerning the relationship between $\sim_0$, $\sim$ and $\FS$. 
    \begin{enumerate}[topsep=5pt,itemsep=5pt,label=\textit{(\arabic*)}]
        \item \label{it:lem:basic-sim0-implies-equal} If $A\sim_0 A'$ then $\FS(A)=\FS(A')$. \
        \item \label{it:lem:basic-sim-implies-shift} If $A\sim A'$, then there is $g\in G$ such that $\FS(A) = \FS(A')+g$.
        \item \label{it:lem:basic-sim-and-fs-implies-sim0} Assume that $G$ does not have elements with order $2$. If $\FS(A)=\FS(A')$ and $A\sim A'$, then $A\sim_0 A'$.
        \item \label{it:lem:basic-shift-implies-equal} If $\FS(A)=\FS(A')+g$ for some $g\in G$, then there exists $\M(G)\ni A''\sim A'$ such that $\FS(A) = \FS(A'')$.
    \end{enumerate}
\end{lemma}
\begin{proof} 
    The following paragraph describes a very simple bijection in a very complicated way, this is due to the formalism necessary to handle the multiplicities of elements in multisets. We suggest the reader to refer to the picture \cref{fig:sim0-same-fs}, which shall be much clearer than the proof itself.
    
    If $A\sim A'$, then, by definition, there is $B\subseteq A$ such that $A'=(A\setminus B)\cup(-B)$.
    So, we have
    \begin{align*}
        \PS(A') &= 
        \Big\{C\cup(-D):\, (C, D)\in \PS(A\setminus B)\times \PS(B)\Big\}
        \\
        &=
        \Big\{C\cup(-(B\setminus D)):\, (C, D)\in \PS(A\setminus B)\times \PS(B)\Big\}
    \end{align*}
    and therefore
    \begin{equation}\label{eq:local3424}
    \begin{aligned}
        \FS(A') 
        &= 
        \Big\{\sum C + \sum D -\sum B:\, (C, D)\in \PS(A\setminus B)\times \PS(B)\Big\}
        \\
        &=
        \Big\{\sum C + \sum D:\, (C, D)\in \PS(A\setminus B)\times \PS(B)\Big\} - \sum B 
        \\
        &= \FS(A)-\sum B.
    \end{aligned}
    \end{equation}
    This proves \cref{it:lem:basic-sim-implies-shift}. 
    
    Notice that if $A\sim A'$ and $\sum B=0$, then \cref{eq:local3424} implies that $\FS(A)=\FS(A')$. Hence also \cref{it:lem:basic-sim0-implies-equal} are proven.

    Let us show \cref{it:lem:basic-sim-and-fs-implies-sim0}. The assumption $\FS(A)=\FS(A')$, together with \cref{eq:local3424}, implies that $\FS(A)=\FS(A)-\sum B$. By taking the sum of the elements of the two multisets, we get
    \begin{equation*}
        \sum(\FS(A)) = \sum(\FS(A)-\sum B) = \sum(\FS(A))-|\FS(A)|\cdot \sum B,
    \end{equation*}
    thus $2^{|A|}\sum B = 0_G$. 
    Since $G$ has no elements of order $2$, we deduce $\sum B = 0_G$ and therefore $A\sim_0 A'$ as desired.
    
    In order to prove \cref{it:lem:basic-shift-implies-equal}, notice that $0_G\in \FS(A)$ and thus $-g\in \FS(A')$; so there is $B\subseteq A'$ such that $\sum B = -g$. Let $A''\sim A$ be the multiset $A''\defeq (A'\setminus B)\cup (-B)$. The formula \cref{eq:local3424} (with $A, A'\to A', A''$) yields $\FS(A'') = \FS(A') - \sum B = \FS(A')+g = \FS(A)$ as desired.
\end{proof}

Let us recall the definition of $\FS$-regular groups already given in the introduction.
\begin{definition}
    An abelian group $G$ is $\FS$-regular if, for any $A, A'\in\M(G)$, it holds $\FS(A)=\FS(A')$ if and only if $A\sim_0 A'$.
\end{definition}
Notice that if $G$ is $\FS$-regular, then also its subgroups are $\FS$-regular.
Moreover, it is always true that $A\sim_0 A'$ implies $\FS(A)=\FS(A')$ (see \cref{lem:basic-properties}-\cref{it:lem:basic-sim0-implies-equal}) and therefore the content of the $\FS$-regularity is the opposite implication, which does not hold for all groups.

As anticipated in the introduction, the main result of our work brings into play a subset $O_{\FS}$ of the natural numbers.
We recall its definition and explore some basic properties of these numbers.
\begin{definition}
    Let $O_{\FS}$ be the set of odd natural numbers $n\ge 1$ such that $(\cyclic{n})^*$ is covered by $\{\pm 2^j: j\ge 0\}$; more precisely, for each $x\in\Z$ relatively prime with $n$ there exists $j\ge 0$ such that either $x-2^j$ or $x+2^j$ is divisible by $n$.
\end{definition}
The sequence of the elements of $O_{\FS}$ greater than $1$ is given by OEIS \href{https://oeis.org/A333854}{A333854}, while the complement (in the odd integers greater than $1$) is \href{https://oeis.org/A333855}{A333855}.

\begin{proposition}
    \label{prop:OFS-stability-for-divisors}
    Let $n$ be an element of $O_{\FS}$. Then, all the positive divisors of $n$ are in $O_{\FS}$ as well.
\end{proposition}
\begin{proof}
    Take a positive divisor $d$ of $n$.
    Let $x \in \Z$ be relatively prime with $d$.
    There exists $m \in \Z$ so that $x + md$ is relatively prime with $n$.
    Since $n \in O_{\FS}$, there exists $j \in \N$ such that either $x + md - 2^j$ or $x + md + 2^j$ is a multiple of $n$, and thus a multiple of $d$ as well. Therefore either $x-2^j$ or $x+2^j$ is a multiple of $d$.
\end{proof}

In the following, we denote by $\ord_n(x)$ the multiplicative order of $x$ in $\cyclic{n}$, and by $\varphi$ Euler's totient function.

\begin{proposition}
    \label{prop:OFS-characterization}
    An odd positive integer $n$ is a member of $O_{\FS}$ if and only if one of the following holds:
    \begin{enumerate}[topsep=5pt,itemsep=5pt,label=(\roman*)]
        \item $\ord_n(2) = \varphi(n)$;
        \item \label{item:OFS-characterization-ii} $\ord_n(2) = \varphi(n) / 2$ and either $4 \nmid \varphi(n)$ or $2^{\varphi(n)/4} \not\equiv -1 \pmod{n}$.
    \end{enumerate}
\end{proposition}
\begin{proof}
    Notice that, in $\cyclic{n}$, it holds
    \begin{equation}
        \label{eq:powers-of-2-inequality}
        |\{\pm 2^j : j \ge 0\}| \le 2|\{2^j : j \ge 0\}| = 2\ord_n(2).
    \end{equation}
    Thus, if $n \in O_{\FS}$ then necessarily $\ord_n(2) \ge \varphi(n) / 2$.

    If $\ord_n(2) = \varphi(n)$, then $\{\pm 2^j : j \ge 0\} = \{2^j : j \ge 0\} = (\cyclic{n})^*$.

    If $\ord_n(2) = \varphi(n) / 2$, then in order to have equality in \eqref{eq:powers-of-2-inequality} it is necessary and sufficient that $2^j \not\equiv -2^{j'} \pmod{n}$ for all $j, j' \ge 0$, which is equivalent to $2^j \not\equiv -1 \pmod{n}$ for all $j \ge 0$.
    If $4 \nmid \varphi{n}$, the latter is impossible.
    Otherwise, the only $0 \le j < \varphi(n)$ for which the congruence can be true is $\varphi(n) / 4$, hence \ref{item:OFS-characterization-ii} follows.
\end{proof}

\begin{proposition}
    \label{prop:OFS-2-primes}
    If $n \in O_{\FS}$, then $n$ is divided by at most two distinct primes.
\end{proposition}
\begin{proof}
    In view of \cref{prop:OFS-stability-for-divisors} and working by contradiction, it is sufficient to show that $pqr \not\in O_{\FS}$ whenever $p$, $q$, $r$ are distinct odd primes.

    Since $q - 1$ and $r - 1$ are even, $p - 1$ divides $h := (p - 1)(q - 1)(r - 1) / 4$, and thus $2^h \equiv 1 \pmod{p}$.
    Likewise, $2^h \equiv 1$ modulo $q$ and $r$.
    This implies that $2^h \equiv 1 \pmod{pqr}$, therefore $\ord_{pqr}(2) \mid h = \varphi(pqr) / 4$.
    By \cref{prop:OFS-characterization}, we deduce that $pqr \not\in O_{\FS}$.
\end{proof}

One might wonder whether it is true that every $n \in O_{\FS}$ has at least one multiple in $O_{\FS}$.
This is false: a counterexample is $3p$ with $p=3511$.
It can be verified that $\ord_{3p}(2) = p-1$ and $2^{(p-1)/2}\not\equiv-1\pmod{3p}$, and thus $3p \in O_{\FS}$ thanks to \cref{prop:OFS-characterization}.
\cref{prop:OFS-2-primes} tells us that any multiple of $3p$ that belongs to $O_{\FS}$ must be of the form $3^a \cdot p^b$, so it is enough to check that $9p$ and $p^2$ are not in $O_{\FS}$, which is true (since $\ord_{9p}(2)=p-1$ and $\ord_{p^2}(p)=(p-1)/2$).
The number $p=3511$ is a Wieferich prime (cf. \cite{CrandallDilcherPomerance1997}), that is, a prime $p$ such that $p^2$ divides $2^{p - 1} - 1$ (and, in fact, one of the only two known such primes).
It is natural to use a Wieferich prime $p$ in this construction because, even if $p\in O_{\FS}$, the fact that $\ord_{p^2}(2) \mid p - 1$ guarantees that $p^2\not\in O_{\FS}$.

\section{\texorpdfstring{$\FS$}{FS}-regularity of cyclic groups}\label{sec:finite-cyclic}
In this section we characterize the $\FS$-regular finite cyclic groups; the two main results are \cref{prop:fs-regular-implies-ofs} and \cref{prop:ofs-implies-fs-regular}.

To show that if $n\not\in O_{\FS}$ then $\cyclic{n}$ is not $\FS$-regular we produce an explicit counterexample.

\begin{proposition}\label{prop:fs-regular-implies-ofs}
    For any $n\not\in O_{\FS}$, the group $\cyclic{n}$ is not $\FS$-regular.
\end{proposition}
\begin{proof}
    If $n$ is even, then $\cyclic{2}$ is a subgroup of $\cyclic{n}$ and thus it is sufficient to show that $\cyclic{2}$ is not $\FS$-regular. 
    As a counterexample to $\FS$-regularity in $\cyclic{2}$, it is enough to notice that
    \begin{equation*}
        \FS(\{0, 1\}) = \{0, 0, 1, 1\} = \FS(\{1, 1\}),
    \end{equation*}
    while $\{0,1\}\not\sim_0\{1, 1\}$ as multisets with values in $\cyclic{2}$.
    
    Let us now consider the case of $n$ odd.
    Since $n\not\in O_{\FS}$, there exists $k \in (\cyclic{n})^* \setminus \{\pm 2^j \bmod  n\}_{j \in \N}$.
    Moreover, let $d\defeq\varphi(n)$ be so that $n \mid 2^d - 1$.
    Consider the multisets $A,A'\in\M(\cyclic{n})$ defined as
    \begin{equation*}
        A \defeq \{2^0, \, 2^1, \, \dots, \, 2^{d - 1}\} \quad \textnormal{and} \quad
        A' \defeq k\cdot A = \{2^0k, \, 2^1k, \, \dots, \, 2^{d - 1}k\}.
    \end{equation*}
    The choice of $k$ implies that $A \cap A' = (-A) \cap A' = \varnothing$ and, in particular, $A \not\sim_0 A'$.

    We have that\footnote{The unions are taken over $\frac{2^d - 1}{n}$ copies of the same multiset and shall be interpreted in the multiset sense, so that the result is a multiset where each element
    appears $\frac{2^d - 1}{n}$ times.}
    \begin{align*}
        \FS(A) & = \{0, 1, \, 2, \, \dots, \, 2^d - 1\}
        = \{0\}\cup\bigcup_{i = 1}^{\frac{2^d - 1}{n}} \{0, \, 1, \, \dots, \, n - 1\} \\
        & = \{0\}\cup\bigcup_{i = 1}^{\frac{2^d - 1}{n}} k\cdot\{0, \, 1, \, \dots, \,  n - 1\}
        = \{k\cdot 0, \, k \cdot 1, \, k \cdot 2, \, \dots, \, k \cdot (2^d - 1)\} \\
        & = \FS(A').
    \end{align*}
\end{proof}

The proof that $\cyclic{n}$ is $\FS$-regular when $n\in O_{\FS}$ is more involved. The rest of this section is devoted to establish this result by reducing it to a statement about the units of the cyclotomic field $\Q(\omega_n)$. 

Before delving into the proof, let us present the relation between the problem at hand and the units of the cyclotomic field $\Q(\omega_n)$, to clarify the importance of \cref{def:Kn,def:phi}.

Given two multisets $A, A'\in\M(\cyclic{n})$, the condition $\FS(A)=\FS(A')$ is equivalent to the polynomial identity
\begin{equation*}
    \prod_{a\in A} (1+t^a) = \prod_{a'\in A'} (1+t^{a'}) \pmod{t^n-1},
\end{equation*}
which is equivalent to
\begin{equation*}
    \prod_{j=0}^{n-1} (1+\omega_d^j)^{\mu_A(j)-\mu_{A'}(j)} = 1,
\end{equation*}
for all divisors $d\mid n$ (because a polynomial is divisible by $t^n-1$ if and only if it has $\omega_d$ as root for all divisors $d\mid n$). Therefore, we are interested in the kernel of the map which takes a vector $x\in\Z^n$ and produces the tuple, indexed by the divisors $d\mid n$,
\begin{equation}\label{eq:local_phi_intro23}
    \Big(\prod_{j=0}^{n-1} (1+\omega_d^j)^{x_j}\Big)_{d\mid n}.
\end{equation}
Since this map is a homomorphism between abelian groups, studying its kernel is tightly linked to the study of its image. In fact, the crux of this section is the determination of the image of such map (see \cref{lem:phi-has-full-rank}).

The multiplicative group generated by $1+\omega_d^0, 1+\omega_d^1, \dots, 1+\omega_d^{n-1}$ is introduced in \cref{def:Kn}, while its rank is computed in \cref{lem:Kn-rank} (the assumption $n\in O_{\FS}$ is necessary to compute the rank). Then, in \cref{def:phi} we introduce the notation that allows studying the map mentioned in \cref{eq:local_phi_intro23} and we go on 
 to prove its \emph{moral} surjectivity (i.e., its image has full rank) in \cref{lem:phi-has-full-rank} (notice that we do not need $n\in O_{\FS}$, $n$ being odd suffices).
Finally, in \cref{prop:ofs-implies-fs-regular}, we join all the pieces to obtain the desired result.

Given an odd positive integer $n$, recall that, for $1\le j<n$, $1+\omega_n^j$ is a unit of $\Q(\omega_n)$ (see \cref{subsec:intro_cyclo}).

\begin{definition}\label{def:Kn}
    Given an odd positive integer $n\ge 1$, let $K_n$ be the multiplicative subgroup of $\C$ generated by $\{1 + \omega_n^j : 0 \le j < n\}$.
    Note that we include $1 + \omega_n^0 = 2$ among the generators.
\end{definition}

\begin{lemma}\label{lem:Kn-rank}
    If $n \ge 3$ and $n\in O_{\FS}$, it holds $\rk(K_n) = \frac{\varphi(n)}{2}$, where $\varphi$ denotes Euler's totient function.
    Moreover, it holds $\rk(K_1) = 1$.
\end{lemma}
\begin{proof}
    For $n = 1$, $K_n = \langle 2 \rangle \cong \Z$, which has rank $1$.
    
    Let us now consider $K_n$ for $n\ge 3$ and $n \in O_{\FS}$.
    Notice that all generators of $K_n$ apart from the element $2$ are units of $\Q(\omega_n)$, while the inverse of $2$ is not an algebraic integer. Therefore, one obtains $K_n\cong \langle 2\rangle \oplus \tilde K_n$, where $\tilde K_n\defeq \langle 1 + \omega_n^j : 1 \le j < n \rangle$.
    
    It remains to compute the rank of $\tilde K_n$. We have already observed that $\tilde K_n$ is a subgroup of $C_n$ (defined in the statement of \cref{thm:Cn-def}). Using that $n\in O_{\FS}$ we are going to prove that $C_n$ is a subgroup of $\tilde K_n\cup(-\tilde K_n)$.\footnote{One may check that $-1\not\in \tilde K_7$, while $-1\in C_7$. So it is not true in general that $C_n$ and $\tilde K_n$ coincide. On the other hand, for some values of $n$ (e.g., $n=3, 5, 9$) one has $-1\in\tilde K_n$.}
    
    To show that $C_n\subseteq \tilde K_n\cup(-\tilde K_n)$, it is sufficient to show that all generators of $C_n$ belong to $\tilde K_n$ or to $-\tilde K_n$. 
    Let us fix $s\ge 1$ coprime with $n$. Since $n\in O_{\FS}$, there exists $j\ge 0$ such that $\omega_n^{2^j}=\omega_n^s$ or $\omega_n^{2^j} = \omega_n^{-s}$.
    
    If $\omega_n^{2^j}=\omega_n^s$, then, for any $0<r<n$, we have
    \begin{equation*}
        \frac{1-\omega_n^{rs}}{1-\omega_n^r} = \frac{1-\omega_n^{2^jr}}{1-\omega_n^r}
        = \prod_{k=0}^{j-1} (1 + \omega_n^{2^kr}) \in \tilde K_n.
    \end{equation*}
    To handle the case $\omega_n^{2^j}=\omega_n^{-s}$, let us observe that $\omega_n = \frac{1 + \omega_n}{1 + \omega_n^{-1}} \in \tilde K_n$. Therefore, for any $0<r<n$, we have
    \begin{equation*}
        \frac{1 - \omega_n^{rs}}{1 - \omega_n^r}
        = -\omega_n^{rs} \frac{1 - \omega_n^{2^jr}}{1 - \omega_n^r} \in -\tilde K_n.
    \end{equation*}
    We have shown $\tilde K_n \subseteq C_n\subseteq \tilde K_n\cup (-\tilde K_n)$ and thus $\rk(\tilde K_n) = \rk(C_n) = \varphi(n)/2-1$ (recall \cref{cor:rk-Cn}). Hence we conclude $\rk(K_n) = \rk(\langle 2\rangle \oplus\tilde K_n) = 1 + \rk(\tilde K_n) = \varphi(n)/2$.
\end{proof}

\begin{definition}\label{def:phi}
Given a positive integer $n\ge 1$, for $0\le j<n$, let $e^{n}_j$ be the $j$-th canonical generator of $\Z^n = \bigoplus_{j = 0}^{n - 1} \Z$. The index $j$ of $e^{n}_j$ shall be interpreted modulo $n$, i.e., $e^{n}_j\defeq e^{n}_{j\bmod n}$, when $j\ge n$.

For a positive divisor $d$ of $n$, let $\pi^n_d:\Z^n\to\Z^d$ be the unique homomorphism such that $\pi^n_d(e^{n}_j) \defeq e^{d}_{j}$ ($=e^{d}_{j\bmod d})$ for all $0<j<n$.

Let $F_n:\Z^n\to K_n$ be the unique group homomorphism such that $F_n(e^{n}_j) = 1 + \omega_n^j$ for each $0 \le j < n$; or equivalently
\begin{equation*}
    F_n(x) = F_n(x_0, \, \dots, \, x_{n - 1}) \defeq \prod_{j = 0}^{n - 1} (1 + \omega_n^j)^{x_j}.
\end{equation*}
\end{definition}

\begin{lemma}\label{lem:easy-linear-algebra}
    Let $\mathbb F$ be a field and let $V$ be a $\mathbb F$-vector space. Given a subset $S\subseteq V$, we denote with $\langle S\rangle_\F$ the subspace generated by the elements of $S$.
    
    Given $k$ vectors $v_1, v_2, \dots, v_k\in V$, for any $\lambda\in\mathbb F$ which is not a root of unity (i.e., $\lambda^q\not=1$ for all positive integers $q\ge 1$) and for any function $\sigma:\{1, 2,\dots, k\} \to \{1, 2, \dots, k\}$, we have
    \begin{equation*}
        \langle v_j - \lambda v_{\sigma(j)}:\, 1\le j\le k\rangle_\F
        =
        \langle v_j:\, 1\le j\le k\rangle_\F.
    \end{equation*}
\end{lemma}
\begin{proof}
    We prove the statement by induction on $k$.
    For $k=0$ there is nothing to prove.
    
    If $\sigma$ is not surjective then we can assume without loss of generality that $\sigma(j)\not=k$ for all $1\le j\le k$. Hence, we can apply the inductive hypothesis and obtain
    \begin{equation*}
        \langle v_j - \lambda v_{\sigma(j)}:\, 1\le j\le k-1\rangle_\F
        =
        \langle v_j:\, 1\le j\le k-1\rangle_\F.
    \end{equation*}
    Since $v_{\sigma(k)} \in \langle v_j:\, 1\le j\le k-1\rangle_\F$, we obtain
    \begin{equation*}
        \langle v_j - \lambda v_{\sigma(j)}:\, 1\le j\le k\rangle_\F
        =
        \langle v_1, v_2, \dots, v_{k-1}, v_k - \lambda v_{\sigma(n)}\rangle_\F
        =
        \langle v_j:\, 1\le j\le k\rangle_\F,
    \end{equation*}
    which is what we sought.
    
    If $\sigma$ is surjective, then it must be a permutation. In particular there exists $q\ge 1$ such that $\sigma^q(j) = j$ for all $1\le j\le k$. Thus, for any $1\le \ell \le k$, we have the telescopic sum
    \begin{align*}
        \sum_{i=0}^{q-1} \lambda^i\big(v_{\sigma^i(\ell)} - \lambda v_{\sigma(\sigma^i(\ell))}\big) = (1-\lambda^q)v_{\ell} .
    \end{align*}
    Since $1-\lambda^q\not=0$ by assumption, we deduce that $v_{\ell}\in \langle v_j - \lambda v_{\sigma(j)}:\, 1\le j\le k\rangle_\F$ for all $1 \le \ell\le k$, which implies the statement.
\end{proof}

\begin{lemma}\label{lem:phi-has-full-rank}
    For any odd positive integer $n$, the image of the map $(F_d\circ \pi^n_d)_{d\mid n}:\Z^n\to\oplus_{d\mid n} K_d$ is a finite-index subgroup of  $\oplus_{d\mid n}K_d$.
\end{lemma}
\begin{proof}
    Let us fix a divisor $d$ of $n$. We are going to identify some elements of the kernel of $F_d$, which is equivalent to producing nontrivial relations in $K_d$. For any divisor $p$ of $d$ and any $0\le j<d/p$, let
    \begin{equation*}
        v^d_{p, j} \defeq e^d_{jp} - \sum_{k=0}^{p-1} e^d_{j+kd/p}.
    \end{equation*}
    Thanks to \cref{prop:distribution-relations}, we know that $F_d(v^d_{p,j})=1$ for all prime divisors $p$ of $d$ and all $0<j<d/p$.
    Therefore, we have identified the subspace
    \begin{equation*}
        \Z^d \supseteq D_d \defeq \langle v^d_{p, j}\rangle_{p\mid d\text{ prime},\, 0\le j < d/p} 
    \end{equation*}
    of the kernel of $F_d$. Let us identify with $[\emptyparam]_{D_d}:\Z^d\to\Z^d/D_d$ the projection to the quotient. 
    
    We claim that $\Psi_n\defeq([\pi^n_d]_{D_d})_{d\mid n}:\Z^n\to \bigoplus_{d \mid n} \Z^d/D_d$ has full rank (i.e., the rank of its image coincides with the rank of its codomain).
    This claim implies the desired result since $F_d$ is surjective for all $d$. 
    
    In order to show that $\Psi_n$ has full rank we consider its tensorization with $\Q$ and show that it is surjective as a linear map between $\Q$-vector spaces.
    With a mild abuse of notation, we keep denoting with $(e^d_j)_{0\le j<d}$ the canonical basis of $\Q^d$ and we keep denoting with $D_d$ the $\Q$-subspace generated by $\{v^d_{p, j}\}_{p\mid d\text{ prime},\, 0\le j\le d/p}$.
    
    Thanks to the basic properties of the tensor product, we have $(\Z^d/D_d)\otimes\Q = \Q^d/ D_d$ and the tensorization $\Psi_n\otimes \Q: \Q^n\to \bigoplus_{d\mid n} \Q^d/ D_d$ satisfies $(\Psi_n\otimes\Q)(e^n_j) = ([e^d_j]_{D_d})_{d\mid n} \in \bigoplus_{d \mid n} \Q^d/D_d$ for all $0\le j < n$.

    The following commutative diagram shall clarify all the steps of the proof up to now.
    \begin{equation*}
    \begin{tikzcd}[row sep=40pt, column sep=50pt]
        \Q^n
        \arrow[rr, twoheadrightarrow, 
        "{(\Psi_n \, \otimes \, \Q)(e^n_j) \, = \, ([e^d_j]_{D_d})_{d\mid n}}"]
        &
        &
        \bigoplus_{d\mid n} \Q^d/D_d
        \\
        \Z^n 
        \arrow[r, "(\pi^n_d)_{d\mid n}"] 
        \arrow[rr, bend left=35, "\Psi_n"]
        \arrow[u, hookrightarrow, "\emptyparam\otimes\Q"]
        &
        \bigoplus_{d\mid n} \Z^d
        \arrow[r, "{([\emptyparam]_{D_d})_{d\mid n}}"]
        \arrow[rd, swap, twoheadrightarrow, "(F_d)_{d\mid n}"]
        & \bigoplus_{d \mid n} \Z^d/D_d 
        \arrow[u, swap, "\emptyparam\otimes\Q"]
        \arrow[d, dashed, twoheadrightarrow] \\
        & &\bigoplus_{d \mid n} K_d
    \end{tikzcd}.
    \end{equation*}
    
    To prove the surjectivity of the linear map $\Psi_n\otimes \Q:\Q^n\to \bigoplus_{d\mid n} \Q^d/ D_d$ we show explicitly that the canonical generators of the codomain belong to the image of the map.
    
    Given a subset $S\subseteq\{d\ge 1:\, d\mid n\}$ and an index $0\le j<n$, let $u_{S, j}=(u_{S,j}^d)_{d\mid n}\in\bigoplus_{d\mid n} \Q^d/D_d$ be the element defined by
    \begin{equation*}
        \Q^d/D_d \ni 
        u_{S, j}^d \defeq
        \begin{cases}
            0 & \textnormal{if } d \not\in S, \\
            [e^d_j]_{D_d} & \textnormal{if } d \in S.
        \end{cases}
    \end{equation*}
    The index $j$ of $u_{S, j}$ should be interpreted modulo $n$ (e.g.\ $u_{S, n} = u_{S, 0}$).

    Notice that $(u_{\{d\}, j})_{d\mid n, \, 0\le j<n}$ is a set of generators of $\bigoplus_{d\mid n} \Q^d/D_d$.
    Moreover, it holds $(\Psi_n\otimes\Q)(e^n_j)=u_{\{d\ge 1:\, d\mid n\}, \, j}$.
    
    We say that a set $S$ is \emph{solvable} if $u_{S, j}$ belongs to the image of $\Psi_n\otimes\Q$ for all $0\le j<n$. 
    Thanks to the previous observations, we know that $\{d\ge 1:\, d\mid n\}$ is solvable and that the surjectivity of $\Psi_n\otimes\Q$ is equivalent to the fact that all singletons $\{d\}$ are solvable.
    Notice that if $S\subseteq T\subseteq\{d\ge 1:\, d\mid n\}$ is solvable, then also $T\setminus S$ is solvable. Indeed, if $(\Psi_n\otimes\Q)(x) = u_{S, j}$ and $(\Psi_n\otimes\Q)(y)=u_{T, j}$, then $(\Psi_n\otimes\Q)(y-x) = u_{T\setminus S, \, j}$. 
    Our main tool to show the solvability of a set is the following sub-lemma.
    
    \begin{lemma}\label{lem:top-layer-is-solvable}
        Let $S\subseteq\{d\ge 1:\, d\mid n\}$ be a solvable subset and let $p\mid n$ be a prime number. Let us define\footnote{Here $\upsilon_p(x)$ denotes the $p$-adic valuation of a nonzero integer $x$, i.e.\ the maximum exponent $h \ge 0$ such that $p^h$ divides $x$.} $\upsilon_p(S)\defeq \max_{d\in S}\upsilon_p(d)$ as the maximal $p$-adic valuation of an element of $S$.
        Then, the subset $\{d\in S: \upsilon_p(d) = \upsilon_p(S)\}$ is also solvable.
    \end{lemma}
    \begin{proof}
        Let $S'\defeq \{d\in S: \upsilon_p(d) = \upsilon_p(S)\}$.
        Let $m$ be the minimum common multiple of the elements of $S$. Notice that $\upsilon_p(m) = \upsilon_p(S)$.
        
        If $\upsilon_p(S) = 0$, then $S'=S$ and the statement is obvious. From now on we assume that $\upsilon_p(S) > 0$.
        
        We claim that, for any $0\le j < n$, it holds             
        \begin{equation}\label{eq:crucial-identity}
            u_{S, j} - \frac{1}{p} \sum_{k = 0}^{p - 1} u_{S, j + km/p} = 
            u_{S', j} - \frac{1}{p} u_{S', jp}.
        \end{equation}
        We prove \cref{eq:crucial-identity} by looking at the projections of both sides onto $\Q^d/D_d$ and considering various cases depending on the divisor $d$.
        \begin{itemize}
            \item If $d\not\in S$, then $d\not\in S'$ (since $S'\subseteq S$) and thus we have
            \begin{equation*}
                u^d_{S, j} - \frac{1}{p} \sum_{k = 0}^{p - 1} u^d_{S, j + km/p} = 
                0 =
                u^d_{S', j} - \frac{1}{p} u^d_{S', jp}.
            \end{equation*}
            \item If $d\in S$ and $\upsilon_p(d) < \upsilon_p(S)$, then $d\mid\frac mp$ and therefore $u^d_{S, j+km/p} = [e^d_{j+km/p}]_{D_d} = [e^d_j]_{D_d} = u^d_{S, j}$. Since $\upsilon_p(d) < \upsilon_p(S)$ implies that $d\not\in S'$, we deduce
            \begin{equation*}
                u^d_{S, j} - \frac{1}{p} \sum_{k = 0}^{p - 1} u^d_{S, j + km/p} = u^d_{S, j} - \frac{1}{p} \sum_{k = 0}^{p - 1} u^d_{S, j}
                = 0 =
                u^d_{S', j} - \frac{1}{p} u^d_{S', jp}.
            \end{equation*}
            \item If $d\in S$ and $\upsilon_p(d) = \upsilon_p(S)$, then it holds
            \begin{equation}\label{eq:local_sets}
                \Big\{0,\, \frac mp\bmod d,\, 2\frac mp\bmod d, \dots,\, (p-1)\frac mp\bmod d\Big\}
                =
                \Big\{0,\, \frac dp,\, 2\frac dp, \dots,\, (p-1)\frac dp\Big\}.
            \end{equation}
            To prove the latter identity, notice that for any $0\le k<p$, we have
            \begin{equation*}
                \Big(k\frac mp \bmod d\Big) = 
                \Big(k\frac md \bmod p\Big)\frac dp
            \end{equation*}
            and therefore the identity between sets follows from the fact that $m/d$ is not divisible by $p$.
            
            Exploiting \cref{eq:local_sets} and recalling that $v^d_{p, j}\in D_d$, we obtain
             \begin{align*}
                u^d_{S, j} - \frac{1}{p} \sum_{k = 0}^{p - 1} u^d_{S, \, j + km/p} &= 
                \Big[e^d_j - \frac 1p \sum_{k = 0}^{p - 1} e^d_{j + km/p}\Big]_{D_d} = 
                \Big[e^d_j - \frac 1p \sum_{k = 0}^{p - 1} e^d_{j + kd/p}\Big]_{D_d} \\
                &=
                \Big[e^d_j - \frac 1p (e^d_{jp} - v^d_{p, j})\Big]_{D_d}
                =
                \Big[e^d_j - \frac 1p e^d_{jp}\Big]_{D_d}
                \\
                &=
                u^d_{S', j} - \frac{1}{p} u^d_{S', jp},
            \end{align*}
            where in the last steps we used that $d\in S'$ (which is equivalent to the assumptions $d\in S$ and $\upsilon_p(d) = \upsilon_p(S)$).
        \end{itemize}
        Since we have covered all possible cases, \cref{eq:crucial-identity} is proven.
        
        The set $S$ is solvable, therefore the left-hand side of \cref{eq:crucial-identity} belongs to the image of $\Psi_n\otimes\Q$, and thus also $u_{S', j} - \frac1p u_{S', jp}$ belongs to $\im(\Psi_n\otimes\Q)$ for all $0\le j<n$. \Cref{lem:easy-linear-algebra}, applied with $v_j\defeq u_{S', j}$, $\lambda\defeq 1/p$, and $\sigma(j)\defeq (jp\bmod n)$, guarantees that also $u_{S', j}$ belongs to the image of $\Psi_n\otimes\Q$ for all $0\le j < n$, which proves that $S'$ is solvable as desired.
    \end{proof}
    As a simple consequence of \cref{lem:top-layer-is-solvable}, we claim that if $S$ is solvable, then, for any prime divisor $p$ of $n$ and for any $0\le h \le \upsilon_p(n)$, we have that $\{s\in S: \upsilon_p(s) = h\}$ is also solvable.
    Let us prove it by induction on $h$, starting from $h = \upsilon_p(n)$ and going backward to $h = 0$. 

    If $\{s\in S: \upsilon_p(s)=\upsilon_p(n)\}$ is empty, then it is solvable; otherwise we can apply \cref{lem:top-layer-is-solvable} and obtain again that it is solvable. 
    Now, we assume that $\{s\in S: \upsilon_p(s) = h'\}$ is solvable for $h' > h$. 
    Then, since the difference of solvable sets is solvable, we deduce that $\tilde S\defeq \{s\in S: \upsilon_p(s)\le h\}$ is solvable. If $\{s\in S: \upsilon_p(s) = h\}$ is empty, then it is solvable; otherwise we can apply \cref{lem:top-layer-is-solvable} on the set $\tilde S$ and obtain again that $\{s\in S:\upsilon_p(s) = h\}$ is solvable as desired.
    
    We can now conclude by showing that singletons $\{d\}$ are solvable for each $d\mid n$. 
    This follows directly from the fact that $\{d\ge 1: d\mid n\}$ is solvable and that if $S$ is solvable then $\{s\in S: \upsilon_p(s) = h\}$ is solvable for all prime divisors $p\mid n$ and all $h \ge 0$.
\end{proof}

\begin{proposition}\label{prop:ofs-implies-fs-regular}
    For any $n\in O_{\FS}$, the group $\Z/n\Z$ is $\FS$-regular.
\end{proposition}
\begin{proof}
    Let $A, \, A' \in \M(\cyclic{n})$ be two multisets such that $\FS(A) = \FS(A')$; we shall prove that $A\sim_0 A'$.
    
    By definition of the map $\FS$, it holds the polynomial identity in $\Z[t]/(t^n-1)$
    \begin{equation*}
        \sum_{j=0}^{n-1}\mu_{\FS(A)}(j) t^j 
        \equiv \sum_{s\in\FS(A)} t^s
        \equiv \prod_{a\in A} (1 + t^a)
        \equiv \prod_{j=0}^{n-1} (1+t^j)^{\mu_A(j)}
        \pmod{t^n-1} ,
    \end{equation*}
    Thus the condition $\FS(A)=\FS(A')$ is equivalent to
    \begin{equation*}
        \prod_{j=0}^{n-1} (1+t^j)^{\mu_A(j)} 
        \equiv \prod_{j=0}^{n-1} (1+t^j)^{\mu_{A'}(j)}
        \pmod{t^n-1}.
    \end{equation*}
    For any divisor $d\mid n$, $\omega_d$ is a root of $t^n-1$ and therefore the latter identity implies
    \begin{equation*}
         \prod_{j=0}^{n-1} (1+\omega_d^j)^{\mu_A(j)} 
        = \prod_{j=0}^{n-1} (1+\omega_d^j)^{\mu_{A'}(j)}
    \end{equation*}
    which, recalling \cref{def:phi}, is equivalent to
    \begin{equation*}
        F_d\Big(\pi^n_d\big((\mu_A(j)-\mu_{A'}(j))_{0\le j<n}\big)\Big) = 1.
    \end{equation*}
    We have just shown that the vector $(\mu_A(j)-\mu_{A'}(j))_{0\le j<n}\in\Z^n$ belongs to the kernel of the map $(F_d\circ \pi^n_d)_{d\mid n}:\Z^n\to\oplus_{d\mid n}K_d$. Let us now switch our attention to the study of such kernel.
    
    \begin{figure}[htb]
    \begin{equation*}
        \begin{tikzcd}[row sep=large, column sep=55pt]
            \M(\cyclic{n}) 
            \arrow[rr, hookrightarrow, "{A\mapsto (\mu_A(j))_{0\le j<n}}"] 
            \arrow[d, "\FS"]
            &
            &
            \Z^n
            \arrow[d, "{x \mapsto \prod_{j=0}^{n-1} (1+t^j)^{x_j}}"] 
            \arrow[r, rightarrow, "{(F_d\circ \pi_{n, d})_{d\mid n}}"]
            &
            \oplus_{d\mid n} K_d
            \arrow[d, hookrightarrow]
            \\
            \M(\cyclic{n}) 
            \arrow[r, hookrightarrow]
            &
            \Z^n
            \arrow[r, "\cong", "{x\mapsto \sum_{j=0}^{n-1}x_jt^j}"']
            &
            \frac{\Z[t]}{(t^n-1)}
            \arrow[r, "\cong", "{[q]\mapsto (q(\omega_d))_{d\mid n}}"']
            &
            \oplus_{d\mid n} \Z[\omega_d]
        \end{tikzcd}.
    \end{equation*} 
    \caption{A commutative diagram depicting the relation, explained at the beginning of the proof of \cref{prop:ofs-implies-fs-regular}, between the map $\FS$ and the map $(F_d\circ \pi^n_d)_{d\mid n}$.}
    \end{figure}
    
    Due to basic properties of the rank (see \cref{subsec:abelian-groups}), we have
    \begin{align*}
        \rk\big(\ker((F_d\circ \pi^n_d)_{d\mid n})\big)
        &=
        n -
        \rk\big(\im((F_d\circ \pi^n_d)_{d\mid n})\big)
        =
        n -
        \rk\Big(\bigoplus_{d\mid n}K_d\Big)
        \\
        &=n -
        \sum_{d\mid n} \rk(K_d)
        =
        n-1-\sum_{1<d\mid n}\frac{\varphi(d)}2
        =
        \frac{n-1}2,
    \end{align*}
    where we have used \cref{lem:phi-has-full-rank} and \cref{lem:Kn-rank}.
    
    Let us now exhibit a subgroup $L_n$ of $\Z^n$ which is included in the kernel of $(F_d\circ \pi^n_d)_{d\mid n}$ (in hindsight, it coincides with such kernel).
    Let $L_n\subseteq\Z^n$ be the subgroup\footnote{Notice that $L_n$ is the subgroup generated by the vectors $(\mu_B(j) - \mu_{B'}(j))_{0 \le j < n}$ for any two multisets $B\sim_0 B'$.}
    \begin{equation*}
        L_n\defeq \left\{x\in\Z^n
        :\, 
        \begin{aligned}
        &x_0=0,\, \\
        &x_j+x_{n-j}=0\text{ for all $1\le j\le \frac{n-1}2$}, \, \\
        &\sum_{j=1}^{\frac{n-1}2} j\cdot x_j\text{ is divisible by $n$}
        \end{aligned}
        \right\} .
    \end{equation*}
    For any $d\mid n$ and $x\in L_n$, we have
    \begin{align*}
        F_d(\pi^n_d(x)) 
        &= 
        \prod_{j = 0}^{n-1} (1+\omega_d^j)^{x_j} = 
        \prod_{j=1}^{(n-1)/2} (1+\omega_d^j)^{x_j}(1+\omega_d^{-j})^{-x_j}
        \prod_{j=1}^{(n-1)/2} \omega_d^{j\cdot x_j}
        \\
        &=
        \omega_d^{\sum_{j=1}^{(n-1)/2}j\cdot x_j} = 1,
    \end{align*}
    and this proves that $L_n$ is a subgroup of the kernel of $(F_d\circ \pi^n_d)_{d\mid n}$.
    
    Notice that $\rk(L_n)=\frac{n-1}2 =\rk\big(\ker((F_d\circ \pi^n_d)_{d\mid n})\big)$, so for any $x\in \ker((F_d\circ \pi^n_d)_{d\mid n})$ there exists $\alpha\ge 1$ such that $\alpha x\in L_n$ and therefore $x$ itself must satisfy the first two conditions in the definition of $L_n$, that is
    \begin{equation*}
        \ker((F_d\circ \pi^n_d)_{d\mid n})\big) \subseteq 
        \Big\{
        x\in\Z^n: x_0=0,\, x_j+x_{n-j}=0\text{ for all $1\le j\le \frac{n-1}2$}
        \Big\}.
    \end{equation*}
    
    The latter inclusion, together with the vector $(\mu_A(j)-\mu_{A'}(j))_{0\le j<n}\in\Z^n$ belonging to the kernel we are studying, implies
    \begin{equation*}
        \mu_A(0)=\mu_{A'}(0) \;\; \text{and} \;\;
        \mu_A(j) + \mu_A(n-j) = \mu_{A'}(j) + \mu_{A'}(n-j) \text{ for all $1\le j\le n$,}
    \end{equation*}
    that is equivalent to $A\sim A'$. Finally, we conclude $A\sim_0 A'$ taking advantage of \cref{lem:basic-properties}-\cref{it:lem:basic-sim-and-fs-implies-sim0}.
\end{proof}

\section{Radon transform for finite abelian groups}\label{sec:radon}

In this section we will introduce a Radon transform for finite abelian groups and we will show an inversion formula for it. Then we will apply this tool to upgrade \cref{prop:ofs-implies-fs-regular} to the same statement with $\cyclic{n}$ replaced by $(\cyclic{n})^d$ for an arbitrary $d\ge 1$.

Let us introduce the discrete Radon transform.
\begin{definition}\label{def:radon}
    Let $n, d \ge 1$ be positive integers.
    Given a function $f:(\cyclic{n})^d\to\C$, its Radon transform is the function $Rf = R_{n, d}f:\Hom((\cyclic{n})^d,\, \cyclic{n})\times \cyclic{n}\to\C$ given by
    \begin{equation*}
        Rf(\psi, c) \defeq \sum_{\substack{x\in (\cyclic{n})^d\\ \psi(x) = c}} f(x),
    \end{equation*}  
    for all homomorphisms $\psi:(\cyclic{n})^d\to\cyclic{n}$ and all $c\in \cyclic{n}$.
\end{definition}

We named this transformation Radon transform in analogy with the continuous Radon transform on $\R^n$ \cite{helgason1999} which, given a function $f:\R^d\to\R$, produces another function $Rf$ which takes an $(n-1)$-affine hyperplane and returns the integral of $f$ over such hyperplane. Notice that affine hyperplanes are exactly the fibers of linear functionals $\R^n\to\R$ and thus the continuous Radon transform on $\R^d$ coincides (up to adapting the definition to a non-discrete setting) with our definition if $\cyclic{n}$ is replaced by $\R$.

One may wonder if \cref{def:radon} would work even if $\cyclic{n}$ was replaced everywhere by an arbitrary finite abelian group $G$. Although everything would still hold, it is not appropriate to give such a definition. Indeed, any finite abelian group $G$ is a subgroup of $(\cyclic{n})^k$ for $n, k\ge 1$ (where $n$ is the largest order of an element in $G$). Hence the Radon transform on $G^d$ shall be defined as the restriction of $R_{n, kd}$ to $\Hom(G^d, \, \cyclic{n}) \times \cyclic{n}$; that is, by understanding $G^d$ as a subgroup of $(\cyclic{n})^{kd}$ and using the Radon transform of the latter (which uses homomorphisms with codomain equal to $\cyclic{n}$ instead of $G$; notice that $\cyclic{n}$ is a subgroup of $G$).

In the literature, one can find many definitions of discrete Radon transform:
\begin{itemize}
\item The definition given in \cite{DiaconisGraham1985} (and investigated in \cite{FranklGraham1987,Fill1989,Velasquez1997,DeDeoVelasquez2004}), which boils down to the convolution with the characteristic function of a fixed set, is completely unrelated to ours. 

\item 
The very general definition given in \cite{Bolker1987} coincides with ours for the group $(\cyclic{p})^d$ ($p$ being prime) and in that work it is named \emph{$(d-1)$-planes transform}. The assumptions of the criterion \cite[Theorem 1]{Bolker1987} to establish the existence of an inversion formula of a Radon transform do not hold for our Radon transform (for example for the group $(\cyclic{4})^2$). Let us remark that the $(d-1)$-planes transform defined for $\mathbb F_{p^k}$ does not coincide with our Radon transform on $(\cyclic{p^k})^d$ when $k>1$ (in particular, proving the invertibility of the $(d-1)$-planes transform seems to be considerably easier due to the larger number of symmetries).

\item 
The recent work \cite{ChoHyunMoon2018} defines a Radon transform which is almost equivalent to our discrete Radon transform on $(\cyclic{p})^d$, where $p$ is a prime number. In that paper the Radon transform (which they call \emph{classical Radon transform} to distinguish it from the one of Diaconis and Graham) coincides with the restriction of ours to the homomorphisms $\psi\in\Hom((\cyclic{p})^d,\, \cyclic{p})$ such that $\psi(0, 0, \dots, 0, 1)\not=0$. Due to this restriction, they cannot establish a full inversion formula \cite[Theorem 1]{ChoHyunMoon2018}.

\item In the work \cite{AbouelazIhsane2008}, the authors define a discrete Radon transform on $\Z^d$ which is equivalent to the Radon transform on $\Z^d$ with our notation (if one allows the group to be non-finite in the definition). An inversion formula \cite[Theorem 4.1]{AbouelazIhsane2008} is proven for such discrete Radon transform. Joining the methods of \cite{AbouelazIhsane2008} with ours, it might be possible to produce inversion formulas for the discrete Radon transform on groups $(\cyclic{n}\times\Z)^d$ that are neither finite nor torsion-free. We do not investigate this as it goes beyond the scope of the paper.

\item An alternative definition of discrete Radon transform for finite abelian groups is provided in \cite{Ilmavirta2014}. The \emph{maximal Radon transform} defined in this reference \cite[Section 7.3]{Ilmavirta2014} computes the sum of the function $f$ over all translations of maximal cyclic subgroups of $G$.

It is not hard to check that, for $p$ prime, the \emph{maximal Radon transform} on $(\cyclic{p})^2$ coincides with ours. In this special case, the author proves the invertibility of the Radon transform \cite[Lemma 3.4]{Ilmavirta2014}. In general his definition does not coincide with ours and, in particular, the \emph{maximal Radon transform} is not invertible in many important cases \cite[Propositions 7.2, 7.3]{Ilmavirta2014}.
\end{itemize}

The invertibility of the discrete Radon transform we have defined follows directly from the invertibility of the Fourier transform on finite abelian groups (see \cite[Part I]{Terras1999} for an introduction to the Fourier transform on finite abelian groups) (cf. \cite[Theorem 3.1]{helgason1999}, \cite{Strichartz1982}). The inversion formula one obtains in this way uses all the values of the Radon transform to recover $f(0)$.

The inversion formula we prove is \emph{stronger}, indeed $f(x)$ can be recovered using only the values of the Radon transform on the \emph{hyperplanes containing $x$}, that is from the values of $Rf(\psi, \psi(x))$ for all $\psi\in\Hom(\cyclic{n}^d, \cyclic{n})$. Notice that, since the Radon transform is not surjective onto its codomain, it is not strange that it admits different inversion formulas.

To avoid lengthy formulas, we will use the notation $\Hom_n^d\defeq \Hom((\cyclic{n})^d, \cyclic{n})$.

\begin{definition}
    A function $\lambda:\Hom_n^d\to\C$ is an \emph{inverting function} for the Radon transform on $(\cyclic{n})^d$ if 
    \begin{equation}\label{eq:inverting-function}
        f(0) = \sum_{\psi\in\Hom_n^d} \lambda(\psi)Rf(\psi, 0).
    \end{equation}
    for all functions $f:(\cyclic{n})^d\to\C$.
\end{definition}
Let us remark that if $\lambda$ is an inverting function for the Radon transform on $(\cyclic{n})^d$ then, for all $x\in(\cyclic{n})^d$,
\begin{equation*}
    f(x) = \sum_{\psi\in\Hom_n^d} \lambda(\psi)Rf(\psi, \psi(x)).
\end{equation*}
This identity follows from \cref{eq:inverting-function} applied to the function $\tilde f \defeq f(\emptyparam+x)$.

Thanks to the observation above, the inversion formula stated in \cref{thm:radon-inversion} is equivalent to the fact that the function $\lambda_{n,d}:\Hom_{n,d}\to\Q$, defined by
\begin{equation}\label{eq:def-lambda}
    \lambda_{n,d}(\psi)\defeq \frac{1}{n^{d-1}\varphi(n)}\prod_{p\mid \psi}(1-p^{d-1}),
\end{equation}
is an inverting function for the Radon transform on $(\cyclic{n})^d$.

Let us begin with two simple technical lemmas that will be useful in the proof of the inversion formula.

\begin{lemma}\label{lem:inversion-criterion}
    Let $n, d\ge 1$ be positive integers. 
    A function $\lambda:\Hom_n^d\to\C$ is an inverting function for the Radon transform on $(\cyclic{n})^d$ if and only if it satisfies, for all $x\in (\cyclic{n})^d$,
    \begin{equation*}
        \sum_{\substack{\psi \in \Hom_n^d \\ \psi(x) = 0}} \lambda(\psi) =
        \begin{cases}
            1 & \textnormal{if } x = 0, \\
            0 & \textnormal{otherwise}.
        \end{cases}
    \end{equation*}
\end{lemma}
\begin{proof}
    For any $f:(\cyclic{n})^d\to\C$ and any $\lambda:\Hom_n^d\to \C$, it holds
    \begin{align*}
        \sum_{\psi\in\Hom_n^d} \lambda(\psi) Rf(\psi, 0)
        &= 
        \sum_{\psi\in\Hom_n^d} 
        \lambda(\psi) \sum_{\substack{x\in (\cyclic{n})^d \\ \psi(x)=0}} f(x)
        \\
        &=
        \sum_{x\in (\cyclic{n})^d} f(x) \sum_{\substack{\psi\in\Hom_n^d \\ \psi(x)=0}} \lambda(\psi).
    \end{align*}
    Thanks to this identity, the desired statement follows because $f$ can be chosen arbitrarily.
\end{proof}

In the next lemma we show that inverting functions behave nicely with respect to products.

\begin{lemma}\label{lem:inversion-products}
    Let $m, n, d\ge 1$ be positive integers such that $m$ and $n$ are coprime.
    Let $\lambda_m:\Hom_m^d\to\C$ and $\lambda_n:\Hom_n^d\to\C$ be inverting functions for the Radon transform on $(\cyclic{m})^d$ and $(\cyclic{n})^d$ respectively.

    Let $\pi_m:\cyclic{mn}\to\cyclic{m}$ and $\pi_m^d:(\cyclic{mn})^d\to(\cyclic{m})^d$ be the canonical projections. Define $\pi_n$ and $\pi_n^d$ analogously.
    Let $\iota_m:\Hom_{mn}^d\to\Hom_m^d$ be the map such that, for all $\psi\in\Hom_{mn}^d$, it holds $\iota_m(\psi)\circ\pi_m^d = \pi_m\circ \psi$. Define $\iota_n$ analogously.
    
    The function $\lambda_{mn}:\Hom_{mn}^d\to\C$ defined as
    \begin{equation*}
        \lambda_{mn}(\psi)\defeq \lambda_m(\iota_m(\psi))\lambda_n(\iota_n(\psi))
    \end{equation*}
    is an inverting function for the Radon transform on $(\cyclic{mn})^d$.
\end{lemma}
\begin{proof}
    The map $(\iota_m,\, \iota_n):\Hom_{mn}^d\to\Hom_m^d\times\Hom_n^d$ is an isomorphism (induced by the isomorphism $(\pi_m, \pi_n):\cyclic{mn}\to\cyclic{m}\times\cyclic{n}$).
    Moreover, given $\psi\in\Hom_{mn}^d$ and $x\in(\cyclic{mn})^d$, the condition $\psi(x)=0$ is equivalent to $\iota_m(\psi)(\pi_m^d(x))=0$ and $\iota_n(\psi)(\pi_n^d(x))=0$.

    Thus, for all $x\in(\cyclic{mn})^d$, we have
    \begin{align*}
        \sum_{\substack{\psi\in\Hom_{mn}^d\\ \psi(x)=0}}\lambda_{mn}(\psi)
        &=
        \sum_{\substack{\psi\in\Hom_{mn}^d\\ 
        \iota_m(\psi)(\pi_m^d(x))=0 \\ 
        \iota_n(\psi)(\pi_n^d(x))=0}} \lambda_m(\iota_m(\psi))\lambda_n(\iota_n(\psi))
        \\
        &=
        \bigg(\sum_{\substack{\psi\in\Hom_m^d \\ \psi(\pi_m^d(x))=0}} \lambda_m(\psi)\bigg)
        \bigg(\sum_{\substack{\psi\in\Hom_n^d \\ \psi(\pi_n^d(x))=0}} \lambda_n(\psi)\bigg)
        \\
        &=
        \begin{cases}
            1 & \textnormal{if } \pi_m^d(x) = 0\text{ and } \pi_n^d(x)=0, \\
            0 & \textnormal{otherwise},
        \end{cases}
    \end{align*}
    where in the last step we used that $\lambda_m$ and $\lambda_n$ are inverting functions and we have applied \cref{lem:inversion-criterion}.
    Since $x=0$ if and only if $\pi_m^d(x)=0$ and $\pi_n^d(x)=0$, the identity above implies that $\lambda_{mn}$ is an inverting function for the Radon transform on $(\cyclic{mn})^d$ thanks to \cref{lem:inversion-criterion}.
\end{proof}

We are ready to prove that $\lambda_{n,d}$ (see \cref{eq:def-lambda}) is an inverting function for the Radon transform on $(\cyclic{n})^d$.

\begin{proof}[Proof of \cref{thm:radon-inversion}]
We have already observed that if we can prove the inversion formula for $x=0$, then the general case follows.
Hence, our goal is to prove the inversion formula for $f(0)$.

For any $m,n\ge 1$ coprime, it holds $\lambda_{mn, d}(\psi)=\lambda_{m,d}(\iota_m(\psi))\lambda_{n,d}(\iota_n(\psi))$ (see \cref{lem:inversion-products} for the definition of $\iota_m,\iota_n$). This identity follows from the fact that Euler's totient function satisfies $\varphi(m)\varphi(n)=\varphi(mn)$ and, for a prime $p\mid m$, the condition $p\mid \psi$ is equivalent to the condition $p\mid \iota_m(\psi)$.
Therefore, thanks to \cref{lem:inversion-products}, since any number $n$ can be factored into a product of prime powers, if we are able to prove that $\lambda_{n,d}$ is an inverting function when $n$ is a prime power then the full result follows.

It remains to prove that $\lambda_{n,d}$ is an inverting function for $n=p^k$ prime power. In order to do that we start from a \emph{bad} but simple inversion formula and we exploit some simple symmetries of the Radon transform to upgrade it to the desired inversion formula.

Notice that any character of $(\cyclic{n})^d$ can be represented uniquely as $(\cyclic{n})^d\ni x\mapsto \omega_n^{\psi(x)}\in\C$, with $\psi\in\Hom_n^d$. Hence, by using this bijection between the characters and the homomorphisms, the inversion formula for the Fourier transform on $(\cyclic{n})^d$ (see \cite[Chapter 10, Theorem 2]{Terras1999}) can be stated as
\begin{equation*}
    f(0) = \frac1{n^d}\sum_{\psi\in\Hom_n^d}\hat f(\psi) 
    = \frac1{n^d}\sum_{\psi\in\Hom_n^d}\sum_{x\in(\cyclic{n})^d} f(x)\omega_n^{-\psi(x)}.
\end{equation*}
By definition of $Rf$, the previous identity becomes
\begin{equation*}
    f(0)=\frac1{n^d}\sum_{\psi\in\Hom_n^d}\sum_{0\le c <n}\omega_n^{-c} Rf(\psi, c).
\end{equation*}
Notice that this is already a valid inversion formula for the Radon transform, but not the one we are looking for.

By exploiting the invariance of the Radon transform $Rf(a\psi, ac) = Rf(\psi, c)$, for any $0\le a<n$ coprime with $n$, we can continue the previous identity (recall that $\varphi$ denotes Euler's totient function)
\begin{equation*}
    =\frac{1}{n^d}\sum_{\psi\in\Hom_n^d}\sum_{0\le c <n} Rf(\psi, c) \frac{1}{\varphi(n)}\sum_{a\in(\cyclic{n})^*} \omega_n^{-ac}.
\end{equation*}
To proceed further, we remember that the sum of the primitive roots coincides with the M\"obius $\mu$ function; hence we get
\begin{equation*}
    =\frac1{n^d}\sum_{\psi\in\Hom_n^d}\sum_{g\mid n} \frac{\mu(n/g)}{\varphi(n/g)}\sum_{\substack{0\le c <n\\ \gcd(c, n)=g}} Rf(\psi, c).
\end{equation*}
Now, let us use that $n=p^k$ is a prime power. 
Since $\mu$ is zero when evaluated over non-squarefree numbers, we may assume that $n/g=1$ or $n/g=p$ in the latter formula. Thus we obtain
\begin{equation*}
    =\frac1{p^{kd}}\sum_{\psi\in\Hom_n^d}\Big(Rf(\psi, 0)-\frac{1}{p-1}\sum_{t=1}^{p-1}
    Rf(\psi, tp^{k-1})\Big).
\end{equation*}
Thanks to the identity
\begin{equation*}
    \sum_{t=0}^{p-1}
    Rf(\psi, tp^{k-1})
    = Rf(p\psi, 0),
\end{equation*}
we can continue our long chain of equalities
\begin{equation*}
    = \frac1{p^{kd}}\sum_{\psi\in\Hom_n^d}\Big(\frac{p}{p-1}Rf(\psi, 0)-\frac{1}{p-1}Rf(p\psi, 0)\Big).
\end{equation*}
Notice that we have written $f(0)$ using only the values of the Radon transform over hyperplanes containing $0$.
Let us observe that, for $\psi\in\Hom_n^d$, there can be either $0$ or $p^d$ different $\psi'\in\Hom_n^d$ such that $p\psi'=\psi$, depending on whether $p\mid \psi$ or not (recall that $p\mid \psi$ is equivalent to $p\mid \psi(x)$ for all $x\in(\cyclic{n})^d$).
Thanks to this observation, we obtain that $f(0)$ is equal to
\begin{equation*}
    = \frac1{p^{kd}}
    \sum_{\psi\in\Hom_n^d}Rf(\psi, 0)\Big(\frac{p}{p-1}-\frac{p^d}{p-1}[\,p\mid \psi\,]\Big),
\end{equation*}
where $[\emptyparam]$ denotes the Iverson's bracket. Through some simple algebraic manipulation, we finally deduce
\begin{equation*}
    f(0)=\frac1{n^{d-1}\varphi(n)} \sum_{\psi\in\Hom_n^d}Rf(\psi, 0)
    (1-p^{d-1}[\,p\mid \psi\,]),
\end{equation*}
which is the desired inversion formula for $n=p^k$.
\end{proof}

Let us apply this inversion formula to establish the $\FS$-regularity of the group $(\cyclic{n})^d$ when $n\in O_{\FS}$. The idea is to project through an homomorphism onto $\cyclic{n}$, use the $\FS$-regularity of $\cyclic{n}$ proven in \cref{prop:ofs-implies-fs-regular}, and then recover the $\FS$-regularity of $(\cyclic{n})^d$ thanks to the invertibility of the Radon transform on $(\cyclic{n})^d$.

\begin{proposition}\label{prop:ofs-implies-fs-regular-for-products}
    For any $n\in O_{\FS}$ and any $d\ge 1$, the group $(\cyclic{n})^d$ is $\FS$-regular.
\end{proposition}
\begin{proof}
    For a multiset $B\in\M((\Z/n\Z)^d)$, by definition of the Radon transform on $((\cyclic{n})^d$ (see \cref{def:radon}), one has $R\mu_B(\psi, c) = \mu_{\psi(B)}(c)$ (recall that $\mu_B$ denotes the multiplicity of elements in the multiset $B$, see \cref{subsec:multisets}) for any $\psi\in\Hom_n^d$ and any $c\in\cyclic{n}$. Therefore, the inversion formula of \cref{thm:radon-inversion} (recall also \cref{eq:def-lambda}) implies
    \begin{equation}\label{eq:reconstructing-from-projections}
    \begin{aligned}
        \mu_B(x) =
        \sum_{\psi\in\Hom_n^d} \lambda_{n,d}(\psi) \mu_{\psi(B)}(\psi(x))  ,  
    \end{aligned}
    \end{equation}
    for all $x\in(\cyclic{n})^d$.
    Notice that this formula allows us to reconstruct $B$ given all its projections $\psi(B)$ onto $\cyclic{n}$.

    Take two multisets $A, \, A' \in \M((\cyclic{n})^d)$ such that $\FS(A)=\FS(A')$; our goal is to prove that $A\sim_0 A'$.
    
    For any $\psi \in \Hom_n^d$, it holds $\FS(\psi(A))=\FS(\psi(A'))$
    and therefore, since we have shown that $\cyclic{n}$ is $\FS$-regular in \cref{prop:ofs-implies-fs-regular}, we have $\psi(A) \sim_0 \psi(A')$.
    Thus (we use only $\psi(A)\sim \psi(A')$), we deduce that for any $\psi\in\Hom_n^d$,
    \begin{equation}\label{eq:projections-are-sim}
        \mu_{\psi(A)}(x) + \mu_{\psi(A)}(-x) = \mu_{\psi(A')}(x) + \mu_{\psi(A')}(-x)
    \end{equation}
    for all $x\in(\cyclic{n})^d$.
    
    Joining \cref{eq:reconstructing-from-projections,eq:projections-are-sim}, we obtain
    \begin{align*}
        \mu_A(x) + \mu_A(-x)
        & = \sum_{\psi\in\Hom_n^d} \lambda_{n,d}(\psi) \big(\mu_{\psi(A)}(\psi(x)) + \mu_{\psi(A)}(-\psi(x))\big) \\
        & = \sum_{\psi\in\Hom_n^d} \lambda_{n,d}(\psi) \big(\mu_{\psi(A')}(\psi(x)) + \mu_{\psi(A')}(-\psi(x))\big) \\
        & = \mu_{A'}(x) + \mu_{A'}(-x)
    \end{align*}
    for all $x\in (\cyclic{n})^d$. The latter identity is equivalent to $A\sim A'$, which implies $A\sim_0 A'$ thanks to \cref{lem:basic-properties}-\cref{it:lem:basic-sim-and-fs-implies-sim0}.
\end{proof}

\section{\texorpdfstring{$\FS$}{FS}-regularity of products with \texorpdfstring{$\Z$}{Z}}\label{sec:multiply-by-Z}
In this section we show that multiplying by $\Z$ does not break the $\FS$-regularity of a group (see \cref{prop:multiply-by-Z}).
In order to do it, we will need two technical lemmas. The second one, \cref{lem:shift-regular-iff-regular}, gives a condition equivalent to $\FS$-regularity which comes handy in the proof of the main result of this section.

\begin{lemma}\label{lem:add-subset-sums}
    Let $G$ be an abelian group without elements of order $2$.
    Given three multisets $A, \, A', \, B \in \M(G)$, if $A + \FS(B) = A' + \FS(B)$, then $A = A'$.
\end{lemma}

\begin{proof}
    Let us first prove the result when $B = \{b\}$ is a singleton. We prove the result by induction on the cardinality of $A$.
    
    If $|A|=0$, then $\varnothing = A + \FS(B) = A' + \FS(B)$ and thus $A'=\varnothing$.
    
    To handle the case $|A|>0$, we begin by showing that $A$ and $A'$ have a common element. We argue by contradiction, hence we assume that $A$ and $A'$ are disjoint.

    Take any $a \in A$. We have $a + b \in A + \FS(B) = A' + \{0, b\}$. Since $a\not\in A'$, it must hold $a + b \in A'$. By repeating this argument (swapping the role of $A$ and $A'$ and replacing $a$ with $a+b$) we obtain that $a+2b\in A$. Repeating such argument $k$ times, we obtain that $a + kb \in A$ if $k$ is even, and $a + kb \in A'$ if $k$ is odd.
    Since $A$ and $A'$ are finite, $b$ must have finite order, otherwise the elements $(a + kb)_{k \in \N}$ would be all distinct.
    Let $\ord(b)$ be the order of $b$; by assumption $\ord(b)$ is odd. We have the contradiction $A\ni a = a + \ord(b)b \in A'$; therefore we have proven that $A$ and $A'$ have a common element.
    
    Now pick $\bar a \in A \cap A'$.
    It holds
    \begin{align*}
        (A \setminus \{\bar a\}) + \FS(B) & = (A + \FS(B)) \setminus \{\bar a, \, \bar a + b\} \\
        & = (A' + \FS(B)) \setminus \{\bar a, \, \bar a + b\} = (A' \setminus \{\bar a\}) + \FS(B).
    \end{align*}
    Therefore, by the induction hypothesis, $A \setminus \{\bar a\} = A' \setminus \{\bar a\}$,
    which is equivalent to $A = A'$.

    Let us now treat general multisets $B$. We proceed by induction on the cardinality of $B$; the case $|B|=0$ is trivial and the case $|B|=1$ is already established, so we may assume $|B|>1$.
    
    Pick an element $\bar b\in B$.
    We have
    \begin{equation*}
        A + \FS(B) = (A + \FS(B \setminus \{\bar b\})) + \FS(\{\bar b\}),
    \end{equation*}
    and likewise for $A'$.
    Applying the induction hypothesis for the three multiset $A + \FS(B \setminus \{\bar b\}), A' + \FS(B \setminus \{\bar b\}), \{\bar b\}$, yields the relation $A + \FS(B \setminus \{\bar b\}) = A' + \FS(B \setminus \{\bar b\})$,
    and one more application yields the sought $A = A'$.
\end{proof}

\begin{remark}
    \Cref{lem:add-subset-sums} admits a beautiful short proof by computing the Fourier transform (refer to \cite[Chapter VI]{HewittRoss1979} for an introduction to the Fourier analysis on groups) of the multiplicity functions of the two multisets $A+\FS(B)$ and $A'+\FS(B)$ and using the assumption that $G$ has no elements of order $2$ to deduce that a character $\chi\in \hat G$ cannot take the value $-1$. This proof was suggested to us by Noah Kravitz. We decided to keep the combinatorial proof since it is more in line with the \emph{elementary} spirit of this section.
\end{remark}

\begin{lemma}\label{lem:shift-regular-iff-regular}
    Let $G$ be an abelian group without elements of order $2$.
    The group $G$ is $\FS$-regular if and only if, for all $A, A'\in\M(G)$ such that $\FS(A)=\FS(A')+g$ for some $g\in G$, it holds $A\sim A'$.
\end{lemma}
\begin{proof}
    Assume that $G$ is $\FS$-regular and take $A, A'\in\M(G)$ such that $\FS(A)=\FS(A')+g$ for some $g\in G$. Applying \cref{lem:basic-properties}-\cref{it:lem:basic-shift-implies-equal}, we produce a multiset $A''\in\M(G)$ such that $A''\sim A'$ and $\FS(A)=\FS(A'')$; then we deduce $A\sim_0 A''$ because $G$ is $\FS$-regular.
    So, we get $A\sim_0 A''\sim A'$ which implies $A\sim A'$ by transitivity.
    
    Let us now show the converse. Given $A, A'\in\M(G)$ such that $\FS(A)=\FS(A')$, the condition described in the statement implies $A\sim A'$ which implies $A\sim_0 A'$ thanks to \cref{lem:basic-properties}-\cref{it:lem:basic-sim-and-fs-implies-sim0}. Therefore we have proven the $\FS$-regularity of $G$.
\end{proof}

\begin{proposition}
    \label{prop:multiply-by-Z}
    If $G$ is an $\FS$-regular abelian group, then also $G\oplus\Z$ is $\FS$-regular.
\end{proposition}
\begin{proof}
    We begin by setting up some notation.
    For $B \in \M(G\oplus\Z)$ and $z \in \Z$, define 
    \begin{align*}
        B_{< z} &= \{(g, z') \in B : z' < z\}, \\
        B_{\le z} &= \{(g, z') \in B : z' \le z\}, \\
        B_{= z} &= \{(g, z') \in B : z' = z\}.
    \end{align*}
    
    Let $A, A'\in\M(G\oplus\Z)$ be two multisets such that $\FS(A)=\FS(A')+(\bar g, \bar z)$ for some $\bar g\in G$ and $\bar z\in\Z$; we want to prove that $A\sim A'$. This claim is equivalent to the $\FS$-regularity of $G$ thanks to \cref{lem:shift-regular-iff-regular}.
    
    Up to changing the signs\footnote{Formally, we are substituting $A$ and $A'$ with $\tilde A\defeq (A\setminus A_{<0}) \cup (-A_{<0})$ and $\tilde A'\defeq (A'\setminus A'_{<0}) \cup (-A'_{<0})$. Notice that $A\sim \tilde A$ and $A'\sim\tilde A'$.} of $A_{<0}$ and $A'_{<0}$, we may assume that $A_{<0}=\varnothing$ and $A'_{<0}=\varnothing$.
    We will use repeatedly, without explicitly mentioning it, that the first coordinate of the elements of $A$ and $A'$ is nonnegative. 
    
    Recall that, by assumption, $\FS(A)=\FS(A')+(\bar g, \bar z)$. Since $(0_G, 0)$ belongs to both $\FS(A)$ and $\FS(A')$ (and the first coordinate of all the elements of both multisets is nonnegative), it must be $\bar z=0$. So, it holds $\FS(A)=\FS(A') + (\bar g, 0)$.
    
    We prove, by induction on $z$, that $A_{\le z}\sim A'_{\le z}$ and $\FS(A_{\le z}) = \FS(A'_{\le z}) + (\bar g, 0)$. One can deduce $A\sim A'$ by taking $z$ sufficiently large.
    
    Notice that
    \begin{equation*}
        \FS(A_{=0}) = \FS(A)_{=0} = \FS(A')_{=0} + (\bar g, 0).
    \end{equation*}
    By taking the projection on $G$ of both sides of the latter identity, since $G$ is $\FS$-regular, we can apply \cref{lem:shift-regular-iff-regular} and get $A_{=0}\sim A'_{=0}$. This concludes the first step of the induction, that is $z=0$ (since $A_{=0}=A_{\le 0}$ and $A'_{=0}=A'_{\le 0}$).
    
    For $z\ge 1$, we show that $A_{=z} = A'_{=z}$ which immediately implies, thanks to the inductive assumption, that $A_{\le z}\sim A'_{\le z}$ and $\FS(A_{\le z}) = \FS(A'_{\le z}) + (\bar g, 0)$.
    
    Given a multiset $B\in\M(G\oplus \Z)$ such that $B_{<0}=\varnothing$ (later on $B$ will be a subset of $A$ or $A'$), if $\sum B = (g, z)$ for some $g\in G$ and $z\ge 1$ then either $B= B_{<z}$ or $B = B_{=z}\cup B_{=0}$ and $B_{=z}$ is a singleton.
    Hence, one has
    \begin{align*}
        \FS(A)_{=z} = \FS(A_{<z})_{=z} \cup (A_{=z} + \FS(A_{=0})), \\
        \FS(A')_{=z} = \FS(A'_{<z})_{=z} \cup (A'_{=z} + \FS(A'_{=0})),
    \end{align*}
    and therefore, recalling that $\FS(A)=\FS(A')+(\bar g, 0)$, we get
    \begin{equation}\label{eq:loco23}\begin{aligned}
        \FS(A_{<z})_{=z} &\cup (A_{=z} + \FS(A_{=0})) = \FS(A)_{=z} = \FS(A')_{=z} + (\bar g, 0) \\ 
        &=
        (\FS(A'_{<z})_{=z}+(\bar g, 0)) \cup(A'_{=z} + \FS(A'_{=0}) + (\bar g, 0)) .
    \end{aligned}\end{equation}
    By inductive assumption, $\FS(A_{=0})=\FS(A'_{=0})+(\bar g, 0)$ and $\FS(A_{<z}) = \FS(A'_{<z}) + (\bar g, 0)$; hence \cref{eq:loco23} implies
    \begin{equation*}
        A_{=z} + \FS(A_{=0}) = A'_{=z} + \FS(A_{=0})
    \end{equation*}
    and we deduce $A_{=z}=A'_{=z}$ thanks to \cref{lem:add-subset-sums} (since $G$ is $\FS$-regular it cannot have elements of order $2$, see \cref{prop:fs-regular-implies-ofs}).
\end{proof}

\section{Proof of the Main Theorem}\label{sec:proof-main-thm}
The proof of the main theorem of this paper is routine work now that we have established \cref{prop:fs-regular-implies-ofs,prop:ofs-implies-fs-regular,prop:ofs-implies-fs-regular-for-products,prop:multiply-by-Z}.

\begin{proof}[Proof of \cref{thm:main}]
    If there is a torsion element $g\in G$ such that $\ord(g)\not\in O_{\FS}$, then $\cyclic{\ord(g)}$ is a subgroup of $G$. Thanks to \cref{prop:fs-regular-implies-ofs}, we know that $\cyclic{\ord(g)}$ is not $\FS$-regular and therefore also $G$ is not $\FS$-regular.
    
    We prove the converse implication in three steps: first for groups with structure $(\Z\oplus\Z/n\Z)^d$, then for finitely generated groups, and finally for any group.
    
    Let us assume that $G$ is an abelian group such that $\ord(g)\in O_{\FS}$ whenever $g\in G$ has finite order.
    
    \vspace{0.5em}\noindent
    {\textbf{Step 1: $G=(\Z\oplus \cyclic{n})^d$.}} The assumption on the order of the elements of $G$ guarantees that $n\in O_{\FS}$. Hence, \cref{prop:ofs-implies-fs-regular-for-products} shows that $(\cyclic{n})^d$ is $\FS$-regular. 
    Thanks to \cref{prop:multiply-by-Z}, we obtain that also $(\cyclic{n})^d\oplus \Z^d$ is $\FS$-regular.
    
    \vspace{0.5em}\noindent
    {\textbf{Step 2: $G$ is finitely generated.}} Let $n$ be the maximum order of an element in $G$ with finite order. By assumption $n\in O_{\FS}$. The classification of finitely generated abelian groups (see \cref{subsec:abelian-groups}) guarantees that $G$ is a subgroup of $(\Z\oplus\cyclic{n})^d$ for some $d \ge 1$. By the previous step, we know that $(\Z\oplus\cyclic{n})^d$ if $\FS$-regular and thus also $G$ is $\FS$-regular (being a subgroup of an $\FS$-regular group).
    
    \vspace{0.5em}\noindent
    {\textbf{Step 3: No restrictions on $G$.}} Let $A, A'\in \M(G)$ be two multisets such that $\FS(A)=\FS(A')$; we want to prove that $A\sim_0 A'$. Let $\tilde G\defeq \langle A\cup A'\rangle$ be the group generated by the elements of $A$ and $A'$. The condition on the orders is inherited by $\tilde G$ and, since $\tilde G$ is finitely generated, the previous step guarantees that $\tilde G$ is $\FS$-regular; in particular $A\sim_0 A'$ as desired.
\end{proof}

\printbibliography

\end{document}